\documentclass[leqno,11pt]{article}

\usepackage{amsthm,amsfonts,amssymb,amsmath,amscd}
\usepackage[all]{xy}
\usepackage{xr-hyper}
\usepackage[colorlinks=true]{hyperref}

\usepackage{tikz,tikz-cd}
\usetikzlibrary{matrix,arrows,decorations.pathmorphing}

\usepackage[OT2,OT1]{fontenc}

\usepackage{authblk}

\font\cyr=wncyr10

\usepackage{mathrsfs}

\newcommand{\BA}{{\mathbb {A}}}

\newcommand{\BC}{{\mathbb {C}}}

\newcommand{\BF}{{\mathbb {F}}}

\newcommand{\BQ}{{\mathbb {Q}}}
\newcommand{\BR}{{\mathbb {R}}}

\newcommand{\BT}{{\mathbb {T}}}

\newcommand{\BZ}{{\mathbb {Z}}}

\newcommand{\BZp}{{\BZ_p}}
\newcommand{\BQp}{{\BQ_p}}

\newcommand{\CA}{{\mathcal {A}}}

\newcommand{\CG}{{\mathcal {G}}}

\newcommand{\CL}{{\mathcal {L}}}

\newcommand{\CO}{{\mathcal {O}}}

\newcommand{\CS}{{\mathcal {S}}}
\newcommand{\CT}{{\mathcal {T}}}

\newcommand{\CV}{{\mathcal {V}}}

\newcommand{\CX}{{\mathcal {X}}}

\newcommand{\fkL}{{\mathfrak{L}}}

\newcommand{\sA}{{\mathscr{A}}}

\newcommand{\sL}{{\mathscr{L}}}

\newcommand{\sM}{{\mathscr{M}}}

\newcommand{\sO}{{\mathscr{O}}}

\newcommand{\sV}{{\mathscr{V}}}

\newcommand{\fkh}{{\mathfrak{h}}}

\newcommand{\fkm}{{\mathfrak{m}}}

\newcommand{\fkp}{{\mathfrak{p}}}

\newcommand{\fkP}{{\mathfrak{P}}}

\newcommand{\heart}{{\heartsuit}}

\newcommand{\club}{{\clubsuit}}

\newcommand{\spade}{{\spadesuit}}

\newcommand{\Aut}{{\mathrm{Aut}}}

\newcommand{\diag}{{\mathrm{diag}}}

\newcommand{\End}{{\mathrm{End}}}

\newcommand{\Frob}{{\mathrm{Frob}}}

\newcommand{\Gal}{{\mathrm{Gal}}}
\newcommand{\GL}{{\mathrm{GL}}}

\newcommand{\Hom}{{\mathrm{Hom}}}

\newcommand{\ord}{{\mathrm{ord}}}

\newcommand{\Pic}{\mathrm{Pic}}

\newcommand{\Red}{{\mathrm{Red}}}
\newcommand{\Reg}{{\mathrm{Reg}}}

\newcommand{\loc}{{\mathrm{loc}}}

\newcommand{\Sel}{{\mathrm{Sel}}}

\newcommand{\Sha}{{\hbox{\cyr X}}}

\newcommand{\SL}{{\mathrm{SL}}}

\newcommand{\Sp}{{\mathrm{Sp}}}

\newcommand{\tr}{{\mathrm{tr}}}

\newcommand{\ur}{{\mathrm{ur}}}

\newcommand{\cyc}{{\mathrm{cyc}}}

\newcommand{\im}{{\mathrm{im}}}

\newcommand{\lth}{{\mathrm{lg}}}

\newcommand{\Ram}{{\mathrm{Ram}}}

\newcommand{\Trace}{{\mathrm{Trace}}}
\newcommand{\Ta}{{\mathrm{Ta}}}

\newcommand{\WD}{{\mathrm{WD}}}

\newcommand{\eps}{{\epsilon}}
\newcommand{\veps}{{\varepsilon}}

\newcommand{\wh}{\widehat}

\newcommand{\pair}[1]{\langle {#1} \rangle}

\newcommand{\ov}{\overline}
\newcommand{\incl}{\hookrightarrow}

\newcommand{\bs}{\backslash}

\newcommand{\isoarrow}{\stackrel{\sim}{\rightarrow}}

\newtheorem{thm}{Theorem}[section]
\newtheorem{cor}[thm]{Corollary}
\newtheorem{lem}[thm]{Lemma}
\newtheorem{prop}[thm]{Proposition}

\newtheorem{theorem}{Theorem}[section]

\newtheorem{conjecture}[theorem]{Conjecture}

\theoremstyle{definition}

\theoremstyle{remark}
\newtheorem{remark}[thm]{Remark}

\numberwithin{equation}{section}

\title{Indivisibility of Heegner points in the multiplicative case}

\author[1]{Christopher Skinner}
\affil[1]{Department of Mathematics\\ Princeton University\\ Fine Hall, Washington Road\\ Princeton, NJ 08544\\ USA}

\author[2]{Wei Zhang}
\affil[2]{Department of Mathematics\\ Columbia University\\ MC 4423\\
2990 Broadway\newline New York, NY 10027\\ USA}
\date{}

\begin{document}


\maketitle

\begin{abstract}
For certain elliptic curves $E$ over $\BQ$ with multiplicative reduction at a prime $p\geq 5$, we prove the $p$-indivisibility of the derived Heegner classes defined with respect
to an imaginary quadratic field $K$, as conjectured
by Kolyvagin. The conditions on $E$ include that $E[p]$ be irreducible and not finite at $p$ and that $p$ split in the imaginary quadratic field $K$, along with 
certain $p$-indivisibility conditions on various Tamagawa factors. 
The proof extends the arguments of the second author for the case where $E$ has good ordinary reduction at~$p$.  
\end{abstract}

\tableofcontents

\section{Introduction}
Let $E/\BQ$ be an elliptic curve of conductor $N$, let $K$ be an imaginary quadratic field of discriminant $-D$ such that $(D,N)=1$, 
and let $p\geq 5$ be a prime. In \cite{Z13} the second-named author showed that if $E$ has good ordinary reduction 
at $p$ and $p\nmid D$, then under suitable hypotheses on the Galois representation $E[p]$ and the bad reduction of $E$ (including the indivisibility by $p$ of appropriate Tamagawa factors) the Kolyvagin system of cohomology classes in $H^1(K, E[p])$ arising from Heegner points on Shimura curves  is non-zero. Consequences of this include the $p$-part of the 
Birch--Swinnerton-Dyer (BSD) formula in the rank one case as well as the fact that $\ord_{s=1}L(E,s)=1$ is equivalent to the $p^\infty$-Selmer group $\Sel_{p^\infty}(E/\BQ)$ having $\BZp$-corank one.  The purpose of this paper is to extend these results to many cases where $E$ has multiplicative reduction at $p$. That is, to 
cases where
$$
p\mid\mid N.
$$
The cases where we succeed in doing so are those where $p$ splits in $K$, the mod $p$ Galois representation $E[p]$ is not finite at $p$, and where - when $E$ has split multiplicative
reduction at $p$ - the $p$-adic Mazur--Tate--Teitelbaum $\fkL$-invariant of $E$ has valuation equal to $1$. The latter condition is used to 
ensure that the corresponding $\fkL$-invariants of newforms congruent to $E$ are also non-zero.

As an application of our main result for elliptic curves with multiplicative reduction at $p$ we prove:

\begin{thm}\label{thm E/Q}
Let $E/\BQ$ be an elliptic curve with conductor $N$ and minimal discriminant $\Delta$ and let $p\geq 5$ be a prime. Suppose
\begin{itemize}
\item[\rm (a)] $E$ has multiplicative reduction at $p$ (equivalently, $p\mid\mid N$);
\item[\rm (b)] $p\nmid \ord_p(\Delta)$, and if $E$ has split multiplicative reduction at $p$ then $\log_p q_E \in p\BZ_p^\times$,
where $q_E\in \BQ_p^\times$ is the Tate period of $E/\BQ_p$;
\item[{\rm (c)}] $E[p]$ is an irreducible $\Gal(\overline\BQ/\BQ)$-module;
\item[{\rm (d)}] for all primes $\ell\mid\mid N$ such that $\ell\equiv \pm 1\pmod p$, $p\nmid\ord_\ell(\Delta)$;
\item[{\rm (e)}] there exist at least two prime factors $\ell\mid\mid N$ such that $p\nmid\ord_\ell(\Delta)$;
\item[{\rm (f)}] the $p^\infty$-Selmer group $\Sel_{p^\infty}(E/\BQ)$ has $\BZ_p$-corank one.
\end{itemize}
Then the rank and analytic rank of $E/\BQ$ are both equal to~$1$ and the Tate-Shafarevich group 
$\Sha(E/\BQ)$ is finite.
\end{thm}

The first condition in (b) of Theorem \ref{thm E/Q} is equivalent to $E[p]$ not being finite as a representation of $\Gal(\overline\BQ_p/\BQ_p)$,
and - assuming the first condition - the second condition in (b) is then equivalent to the 
Mazur--Tate--Teitelbaum $\fkL$-invariant of $E$ belonging to $p\BZ_p^\times$.
The indivisibility condition in (d) and (e) is equivalent to $E[p]$ being ramified at the prime $\ell$.

We also deduce a result in the direction of the Birch--Swinnerton-Dyer formula for elliptic curves
of analytic rank one:

\begin{thm}\label{thm E BSD} Let $E/\BQ$ be an elliptic curve and let $p\geq 5$ be a prime. Suppose hypotheses $\rm{(a)-(e)}$ of
Theorem \ref{thm E/Q} hold and that $\ord_{s=1}L(E,s) = 1$. Then 
$$
\ord_p(\frac{L'(E,1)}{\Omega_E\cdot\Reg(E/\BQ)}) = \ord_p(\#\Sha(E/\BQ) \cdot \prod_{\ell\mid N} c_\ell).
$$
\end{thm}

\noindent Here $\Reg(E/\BQ) = \frac{\langle y,y\rangle_{NT}}{[E(\BQ):\BZ y]^2}$ with $y\in E(\BQ)$ any non-torsion point 
and $\langle y,y\rangle_{NT}$ the canonical N\'eron--Tate height of $y$, $\Omega_E$ is the canonical (N\'eron) period
of $E$, and the $c_\ell$ are the local Tamagawa numbers of $E$ at the primes $\ell$.

Theorems \ref{thm E/Q} and \ref{thm E BSD} are both proved by studying the divisibility by $p$ of Heegner points on $E$ coming from suitable Shimura curves.
Let $K$ be as above and suppose that $p$ splits in $K$. 
Let $\ov\rho_{E,p}:\Gal(\ov\BQ/\BQ)\rightarrow\mathrm{Aut}_{\BF_p}E[p]$ denote the Galois representation on the
$p$-torsion $E[p]$ of $E$. 
Let $\Ram(\ov{\rho}_{E,p})$ be the set of primes $\ell\mid\mid N$, $\ell\neq p$, such that $\ov\rho_{E,p}$ is ramified at $\ell$ (equivalently, $p\nmid\ord_\ell(\Delta))$.
Write $N=N^+N^-$ where the prime factors of $N^+$ (resp.~$N^-$) are all split (resp.~inert) in $K$. In particular, $p\mid N^+$. 
 Consider the following hypothesis for $(E,p,K)$:
 \medskip
 
\noindent {\em Hypothesis} $\spade$
\begin{itemize}
\item[(1)]  $N^-$ is squarefree  ($N^-=1$ is allowed). 
\item[(2)]
 $\Ram(\ov{\rho}_{E,p})$ contains all primes $\ell\neq p$ such that $\ell\mid\mid N^+$ and all primes $\ell\mid N^-$ such that 
 $\ell\equiv \pm 1\pmod p$. 
\item[(3)] $\Ram(\ov{\rho}_{E,p})\neq \emptyset$, and either $\Ram(\ov{\rho}_{E,p})$ contains a prime $\ell\mid N^-$ or there are at least two primes factors $\ell\mid\mid N^+$. 
\end{itemize}

\noindent We prove:

\begin{thm} 
\label{thm E main}
Let $E/\BQ$ be an elliptic curve of conductor $N$ and minimal discriminant $\Delta$, and let $p$ be a prime such that 
$p\mid\mid N$. Let $K=\BQ[\sqrt{-D}]$ be an imaginary quadratic field such that $(D,N)=1$. If
\begin{itemize}
\item[\rm (a)] $p\geq 5$ and $p$ splits in $K$;
\item[\rm (b)] $\ov\rho_{E,p}$ is an irreducible $\Gal(\ov\BQ/\BQ)$-representation;
\item[\rm (c)] $\ov{\rho}_{E,p}$ is not finite at $p$, and if $E$ has split 
multiplicative reduction at $p$ then $\log_p q_E \in p\BZ_p^\times$, where $q_E\in \BQ_p^\times$ is the Tate period of $E/\BQ_p$;
\item[\rm (d)]   Hypothesis $\spade$ holds for $(E,p,K)$ with $N^-$ a product of an even number of primes ($N^-~=~1$ is allowed),
\end{itemize}
then
$$\kappa=\{c(n,1)\in H^1(K,E[p]):n \in \Lambda\}\neq\{0\}.$$
In particular, $\kappa^\infty\neq \{0\}.$
\end{thm}

\noindent Here $\kappa$ is the mod $p$ Kolyvagin system arising from Heegner points 
over ray class fields of $K$ on a 
certain Shimura curve associated with the factorization $N=N^+N^-$, 
and $\kappa^\infty$ is the full $p$-adic Kolyvagin system.

The proof of Theorem \ref{thm E main}, which closely follows the proof of the main result of \cite{Z13}, makes use of Heegner points on
the modular abelian varieties associated to newforms of level $Nm $ for suitable square-free integers $m$ and which 
are congruent to the newform associated with $E$.  In fact, Theorem \ref{thm E main} is just a special case of a similar theorem for 
newforms with multiplicative reduction at $p$. We defer the statement of this result to Section \ref{main results}.

The proof of Theorem \ref{thm E main}, really of the more general Theorem \ref{Kconj thm}, follows along the lines of the proof 
of \cite[Thm.~9.1]{Z13}, which is the main result in \cite{Z13}.  Most of this paper is taken up with ensuring that the definitions, constructions,
and crucial ingredients used in \cite{Z13} carry over to the cases considered here. In particular, to successfully follow the strategy in \cite{Z13} we 
must supply a few additional ingredients:
\begin{itemize}
\item We prove a version of Ihara's lemma for Shimura curves when the residual representation is an irreducible 
$\Gal(\ov\BQ/\BQ)$-representation
and reducible but not finite as a representation of $\Gal(\ov\BQ_p/\BQ_p)$.
For the case of modular curves this is already in the literature; we give a proof here for Shimura curves (see \ref{Ihara}).
\item We prove a suitable level-raising result for the newforms considered herein (see \ref{level-raise}). 
This is crucially used to construct elements of the Kolyvagin system.
\item We check that appropriate multiplicity one results hold (see \ref{multone}). 
These are essentially due to Mazur and Ribet \cite{MR} in the case of modular curves and to Helm \cite{He} in the general case (by an argument that depends
on the level-raising result).
\item We verify that the Kolyvagin classes satisfy the required local property at primes above $p$ (see \ref{loc at p}). 
This turns out to be straightforward when the residual representation is not finite at $p$.
\item We verify that the crucial cohomological congruence still holds for the Kolyvagin classes (see \ref{coh cong}). This requires
the new versions of Ihara's lemma and the multiplicity one results.
\item We check that the $p$-part of the BSD formula holds for the level-raised newforms congruent to that 
associated with $E$ (see \ref{BSD0}).  
These forms have Selmer rank $0$ and hence analytic rank $0$, so this this is essentially
a consequence of \cite{SU} and \cite{Smult}. However, the $a(p)=1$ case requires checking that the $\fkL$-invariants of the newforms are non-zero.
\item We verify that the result of Ribet--Takahashi/Pollack--Weston relating congruence numbers and Tamagawa numbers holds 
when $p\mid\mid N$ (see Theorem \ref{thm RT}).
\item We explain that the base case of the induction - the Selmer rank one case - still holds (see \ref{rank one}). 
We also include details about the comparisons of periods and related special value formulas used to prove the base case both in this paper and in \cite{Z13}.
\end{itemize}
After confirming that we have these ingredients at our disposal, the proof of \cite[Thm.~9.1]{Z13} carries over directly,
yielding Theorem \ref{thm E main}. Theorems \ref{thm E/Q} and \ref{thm E BSD} are deduced from Theorem \ref{thm E main} 
just as the analogous
results for the case $p\nmid N$ are deduced in \cite{Z13} from \cite[Thm.~9.1]{Z13}.

Motivation for extending the results of \cite{Z13} to cases of multiplicative reduction comes from 
recent joint work of the authors' with Manjul Bhargava \cite{BSZ-BSDprop}, in which Theorem \ref{thm E/Q} is a key ingredient in a proof that 
at least $66.48\%$ of elliptic curves over $\BQ$, when ordered by naive height, satisfy the rank part of the Birch--Swinnerton-Dyer conjecture.
\medskip

\noindent{\it Acknowledgements.}   The first named author was supported in
part by National Science Foundation Grants~DMS-0758379 and DMS-1301842. 
Much of this paper was written while the first named author was a Moore/Tausky-Todd visiting scholar at Caltech in the Spring of 2014.
The second named author was supported in part by the National Science Foundation Grant~DMS-1301848 and 
a Sloan research fellowship.

\section{Notation, conventions, and some preliminary results}

In this section we fix notation that will be in force throughout this paper. As much as possible we have tried to be consistent with the 
notation in \cite{Z13}. We also include some preliminary results about some of the objects introduced.

\subsection{The prime $p$} Throughout, $p\geq 5$ is a fixed prime.

\subsection{Fields and Galois groups} Let $\ov\BQ$ be a fixed Galois closure of $\BQ$. For a number field $M\subset \ov\BQ$,
let $G_M = \Gal(\ov\BQ/M)$.  Given a number field $M$ and a place $w$ of $M$, let $G_{M_w}\subset G_M$ be a decomposition
group of $M$ (which can be identified with $\Gal(\ov M_w/M_w)$ for some $M$-embedding $\ov\BQ\hookrightarrow \ov M_w$).
For $w$ a finite place, let $I_w\subset G_{M_w}$ be the inertia subgroup and $\Frob_w\in G_{M_w}/I_w$ be the
arithmetic Frobenius. Let $\BF_w$ be the residue field of $w$ and let $\ov\BF_w$ be an algebraic closure of $\BF_w$.
Then there is a natural isomorphism $G_{M_w}/I_w\isoarrow G_{\BF_w} = \Gal(\ov\BF_w/\BF_w)$.

\subsection{Cyclotomic characters}
Let $\veps:G_\BQ\rightarrow\BZ_p^\times$ be the $p$-adic character and let $\chi:G_\BQ\rightarrow\BF_p^\times$ be the mod $p$ reduction
of $\veps$.

\subsection{The imaginary quadratic field $K$}
Let $K\subset \ov\BQ$ be an imaginary quadratic field of discriminant $-D<0$ 
in which $p$ splits, and let $\sO_K$ be the
ring of integers of $K$. For a positive integer $n$, let $K[n]/K$ be the ray class extension of conductor $n$.

Let $\chi_K:(\BZ/D\BZ)^\times\rightarrow\{\pm 1\}$ be the odd primitive quadratic character associated with $K$.

\subsection{Objects associated with newforms and Hypothesis $\club$}\label{newform-gal}
For a newform $g=\sum_{n=1}^\infty a(n) q^n$ of weight $2$, level $N$ (which we always assume satisfies $(N,D)=1$), 
and trivial nebentypus, 
let $F$ be the number field generated by the $a(n)$ and let $\sO$ be its ring of integers. The coefficients $a(n)$ 
generate a possibly non-maximal order $\sO_{0}\subset \sO$. Given a prime $\fkp$ of $\sO$ containing $p$,
let $\fkp_0=\fkp\cap\sO_{0}$.  Let $k=\sO/\fkp$ and $k_{0}=\sO_{0}/\fkp_0$.

Let $A$ be an abelian variety in the isogeny class of $\GL_2$-type abelian varieties associated with $g$. We take
$A$ so that $\sO\hookrightarrow\End_\BQ A$. In this case, the $p$-adic Tate module $\Ta_pA$ (resp.~$A[\fkp^n]$) is a free 
$\sO\otimes\BZp$-module (resp.~$\sO/\fkp^n$-module) of rank two. In particular, $\CV = \Ta_pA\otimes_{\sO}F_{\fkp}$
is a two-dimensional $F_\fkp$-space with continuous $G_\BQ$-action, and $\CT = \Ta_\fkp A=\Ta_pA\otimes_{\sO}\sO_{\fkp}$ is
a $G_\BQ$-stable $\sO_{\fkp}$-lattice. Similarly, $V=A[\fkp]\cong \CT/\fkp\CT$ is a two-dimensional $k$-space
with a continuous $k$-linear $G_\BQ$-action. 

There is a continuous $G_\BQ$-representation $\rho:G_\BQ\rightarrow\Aut_{F_\fkp}\CV$. 
The determinant of $\rho$ is the cyclotomic character $\veps$,  
$\rho$ is unramified at all primes $\ell\nmid Np$, and for such a prime $\ell$ we have $\Trace\rho(\Frob_\ell) = a(\ell)$.
Similarly, the determinant of the two-dimensional $k$-representation $V$ is 
the mod $p$ cyclotomic character $\chi$, $V$ is unramified at all $\ell\nmid Np$, and for such an $\ell$ 
the trace of a Frobenius element $\Frob_\ell$ is just $a(\ell) \pmod{\fkp}$.
The semisimplification $V^{ss}$ of $V$ is defined over $k_0$; this follows from the Brauer-Nesbitt Theorem and the Chebotarev Density Theorem.
We denote by $V_0$ the two-dimensional $k_0$ representation of $G_\BQ$ such that $V^{ss} \cong V_0\otimes_{k_0} k$ 
as $G_\BQ$-representations. Let $\ov\rho:G_\BQ\rightarrow \Aut_{k_0} V_0$ be the $G_\BQ$-action on $V_0$.

We will generally assume that 
$$
\text{$V$ is an irreducible $k$-representation}
$$ 
(in which case it is absolutely irreducible). In this case, $V^{ss}=V$ and so $V\cong V_0\otimes_{k_0}k$ and, in particular, 
$$
\ov\rho\otimes_{k_{0}}k \cong \rho \ (\mathrm{mod} \ {\fkp}).
$$ 
Our main results will generally assume that $\ov\rho$ satisfies the following hypothesis:
\medskip

\noindent{\em Hypothesis $\club$}
\begin{itemize}
\item[(1)] $\bar\rho$ is irreducible;
\item[(2)] the image of $\bar\rho:G_\BQ\rightarrow \Aut_{k_0} V_0 \cong \GL_2(k_0)$ contains a non-trivial unipotent element and at 
least two elements conjugate, respectively, to matrices of the form 
$\diag[a,1]$ and $\diag[b,-1]$, with $a,b\in \BF_p^\times\backslash\{\pm 1\}$.
\end{itemize}

\noindent Note that in order for part (2) of this hypothesis to hold, $p$ must be at least $5$.

\subsection{A local property of $\rho$ and $\ov\rho$}
We record an important local property of these Galois representations.

\begin{lem} \label{ord res lem}
Suppose $p\mid\mid N$. Then $a(p)\in\{\pm 1\}$ and the restriction of $\rho$ to the decomposition group $G_\BQp$ at $p$ satisfies 
$$
\rho|_{G_\BQp} \cong \left(\smallmatrix \epsilon\alpha^{-1} & * \\ 0 & \alpha\endsmallmatrix\right),
$$
where $\alpha$ is the unramified character of $G_\BQp$ such that $\alpha(\Frob_p) = a(p)$.
Similarly,
$$
\ov\rho|_{G_\BQp} \cong \left(\smallmatrix \chi \ov\alpha^{-1} & * \\ 0 & \ov\alpha\endsmallmatrix\right),
$$
where $\ov \alpha = \alpha\pmod{\fkp}$.
\end{lem}

\begin{proof} 
Since $p\mid\mid N$ and $g$ has weight two and trivial nebentypus, $a(p)^2 = 1$ by \cite[Thm.~3(iii)]{Li-75}. 
Hence $a(p)\in\{\pm 1\}$.  In particular, $g$ is ordinary with respect to $\fkp$ in the sense that $a(p)$ is a unit modulo $\fkp$.
The stated property of the restriction of $\rho$ to $G_\BQp$ is then a well-known result (cf.~\cite[Thm.~2.2]{Wiles-padic}). 
We recall a proof here that is based on 
Raynaud's generalization of the Tate curve; various ingredients of this proof will 
be used subsequently.

Since $p\mid\mid N$ and $g$ has trivial nebentypus, 
the abelian variety $A$ has completely toric reduction at $p$ (cf.~\cite[Chap.~2, Prop.~1]{MW}). Let $X$ be the character group of
the torus that is the identity component of the special fibre of the N\'eron model of $A$ over $\BZp$. This torus is split
over an at-most quadratic extension, and so $G_\BQp$ acts on $X$ through the Galois group of an unramified extension
of $\BQp$ of degree at most two. In fact, it acts through the character $\alpha$ since the induced action of
$U_p$ on $X$ is just $\Frob_p$ \cite[Prop.~3.8(ii)]{R90} (see also \cite[Thm.~1.7.6(4)]{Helm-thesis}); note that $U_p$ acts
as the (Atkin-Lehner) involution $-w_p$ on the newform $g$. 
Similarly, let $Y$ be the character group of the connected component of the special fibre of the N\'eron model of the dual abelian
variety $A^\vee$. Both $X$ and $Y$ are $\sO$-modules, locally free of rank one. There is a pairing
$$
j: X\times Y\rightarrow \ov\BQ_p^\times
$$
that is both $\sO$-invariant ($j(a\cdot x, y) = j(x,a\cdot y)$) and $G_\BQp$-invariant,
and an $\sO$-linear $G_\BQp$-invariant uniformization
$$
0\rightarrow X \stackrel{j}{\rightarrow} T(\ov\BQ_p)=\Hom(Y,\ov\BQ_p^\times)\rightarrow A(\ov\BQ_p)\rightarrow 0.
$$
This follows from the theory developed in \cite{McC} and \cite{Mo} (see also \cite{Ray}).
The Tate-module $\Ta_pA$ is then identified as a $G_\BQp$-extension
$$
0\rightarrow \Ta_p T\cong \Hom(Y,\BZp(1)) \rightarrow \Ta_pA \rightarrow X\otimes\BZp \rightarrow 0.
$$
These are all free $\sO\otimes\BZp$-modules, the left and right of rank one and the middle of rank two.
Tensoring with $\sO_{\fkp}$ we obtain the $\fkp$-Tate module $\CT$ of $A$ as a $G_\BQp$-extension
$$
0\rightarrow \sO_{\fkp}(\epsilon\alpha^{-1}) \rightarrow \Ta_\fkp A = \CT \rightarrow \sO_{\fkp}(\alpha) \rightarrow 0.
$$
The first claim of the lemma follows since $\CV = \Ta_\fkp A\otimes_\BZp\BQp$. 
The claim for $\ov\rho|_{G_\BQp}$ follows by reducing modulo $\fkp$. 
\end{proof}

\subsection{Necessary and sufficient condition for $\ov\rho$ to be finite at $p$}

Recall that a $G_{\BQ}$-representation of finite order is {\it finite at $p$} if: as a $G_{\BQp}$-representation it is equivalent to the 
representation on the $\ov\BQ_p$-points of a finite flat group scheme over $\BZp$.

\begin{lem} 
Let $\Phi$ be the component group of the N\'eron model $\sA/\BZp$ of $A/\BQp$. 
Then $A[\fkp]$ is finite at $p$ if and only if $\Phi[\fkp]\neq 0$.
\end{lem}

\begin{proof} 
Let $A[\fkp]^f$ (resp.~$A^0[\fkp]^f$) be the $\ov\BQ_p$-points of the the maximal finite flat subgroup scheme of $\sA[\fkp]$
(resp.~$\sA^0[\fkp]$). Note that $A[\fkp]^f$ and $A^0[\fkp]^f$ are naturally subgroups of $A[\fkp]$: the first is the subgroup of 
points that extend to $\ov\BZ_p$-points on $\sA$ and the second is the subgroup of such points that reduce to the connected component of the identity on the special fibre. 
Then $\Phi[\fkp]\cong A[\fkp]^f/A^0[\fkp]^f$ 
(even as finite flat group schemes). 
Since $p\mid\mid N$ and $g$ has trivial nebentypus, the reduction of $A$ at $p$ is purely toric (cf.~\cite[Chap.~2, Prop.~1]{MW}):
the connected component of the identity of the special fibre of $\sA$ is a torus $T$
over $\BF_p$. The dimension of this torus equals the dimension of $A$, which is $[F:\BQ]$, and there is a faithful action of
$\sO$ on $T$. It follows that $A^0[\fkp]^f\otimes\BF_p$ has rank $[k:\BF_p]$ as a group scheme over $\BF_p$ and hence
that $A^0[\fkp]^f$ is a one-dimensional $k$-space. Therefore,
$$
\Phi[\fkp]\neq 0 \iff A[\fkp]^f/A^0[\fkp]^f \neq 0 \iff \dim_k A[\fkp]^f\geq 2.
$$
Since $A[\fkp]^f\subset A[\fkp]$ and $\dim_k A[\fkp] =2$, it follows that 
$$
\Phi[\fkp]\neq 0 \iff A[\fkp]^f = A[\fkp].
$$
The lemma follows.
\end{proof}

Suppose $V$ is irreducible (equivalently, $\ov\rho$ is irreducible). Since $V_0\otimes_{k_0} k  \cong V = A[\fkp]$ in this case, we 
then also have:

\begin{cor}\label{finite at p}
Suppose $\bar\rho$ is irreducible. Then $\bar\rho$ is finite at $p$ if and only if $\Phi[\fkp]\neq 0$.
\end{cor}

\noindent This follows directly from the preceding lemma as, clearly, $V_0$ is finite at $p$ if and only if $V$ is.

\subsection{Split multiplicative reduction, the $\fkL$-invariant, and Hypothesis $\fkL$}\label{L-inv}
Let $g$ and $\fkp$ be as in \ref{newform-gal}. Suppose $p\mid\mid N$ and $a(p) = 1$. In this case, we say that $g$ has
{\em split multiplicative reduction}, following the terminology for elliptic curves. 

In \cite{MTT} Mazur, Tate, and Teitelbaum defined
an $\fkL$-invariant $\fkL(g) = \fkL(\CV)\in F_{\fkp}$ for $g$. We recall this here.
Returning to the notation of the proof of Lemma \ref{ord res lem}, composition of $j$ with $\ord_p$ induces a non-degenerate pairing
$$
\alpha_p: X\otimes\BQ\times Y\otimes\BQ\stackrel{\ord_p\circ j}{\rightarrow}\BQ.
$$
Similarly, composition of $j$ with the $p$-adic logarithm\footnote{We take
this to be the Iwasawa branch: $\log_p p = 0$.} gives another pairing
$$
\beta_p: X\otimes\BQp\times Y\otimes\BQ_p\stackrel{\log_p\circ j}{\rightarrow}\BQ_p.
$$
As $X\otimes\BQp$ and $Y\otimes\BQp$ are both free $\sO\otimes\BQp = F\otimes\BQp$-modules of rank one,
there exists an element $\fkL\in F\otimes\BQp$ such that $\beta_p = \fkL\cdot \alpha_p$. Then 
$\fkL(\CV)\in F_{\fkp}$ is defined to be the $\fkp$-component of $\fkL$. That is, $\fkL(\CV)$ is the image of $\fkL$ under the projection
$F\otimes\BQp \twoheadrightarrow F_{\fkp}$.

As explained by Greenberg and Stevens \cite[\S3]{GrSt}, the
$\fkL$-invariant can also be defined as follows. 
We have
$$
H^1(\BQp,F_{\fkp}) = \Hom_{cts}(G_\BQp,F_{\fkp}) = \Hom_{cts}(G_\BQp^{ab,p},F_{\fkp}),
$$
where $G_{\BQp}^{ab,p}$ is the maximal abelian pro-$p$ quotient of $G_{\BQp}$. Local class field theory
(normalized so that the reciprocity law takes uniformizers to arithmetic Frobenius elements)
gives an identification
$$
\varprojlim_{n} \BQ_p^\times/(\BQ_p^\times)^{p^n} \isoarrow G^{ab,p}.
$$
Let $u\in 1+p\BZp$ be a topological generator. 
From the decomposition $\BQ_p^\times = p^\BZ \times \BZ_p^\times$
we then obtain an $F_{\fkp}$-basis $\{\psi_\ur$, $\psi_\cyc\}$ of $H^1(\BQp,F_{\fkp})$
with 
$$
\psi_\ur(p) = 1 = \psi_\cyc(u) \ \ \text{and} \ \ \psi_\ur(u)=0 = \psi_\cyc(p).
$$

By Lemma \ref{ord res lem}, as a $G_\BQp$-representation $\CV$ can be realized as an extension
$$
0\rightarrow F_{\fkp}(1) \rightarrow \CV \rightarrow F_{\fkp} \rightarrow 0.
$$
This extension is well-defined up to isomorphism.
Let $c\in H^1(\BQp,F_{\fkp}(1))$ be the class associated with this extension; this is well-defined up to $F_{\fkp}^\times$-multiple.
Kummer theory gives an identification
$$
(\varprojlim_{n}\BQ_p^\times/(\BQ_p^\times)^{p^n})\otimes_\BZp F_{\fkp} \isoarrow H^1(\BQp,F_{\fkp}(1)).
$$
As $p\mid\mid N$, the $G_\BQp$-representation $\CV$ is semistable but not crystalline: this follows from the previously
made observation that $A$ has purely toric reduction and the description of 
$\CV$ as a $G_\BQp$-representation in the proof of Lemma \ref{ord res lem}.
In particular, 
$c$ is not identified with an element of $(\varprojlim_n \BZ_p^\times/(\BZ_p^\times)^{p^n})\otimes_\BZp F_{\fkp}$, which is the 
subspace of crystalline extensions \cite[Ex.~3.9]{BlochKato}, and so $\psi_\ur(c)\neq 0$. The $\fkL$-invariant of $\CV$ is then just
$$
\fkL(g) = \fkL(\CV) = \log_pu \cdot \psi_\ur(c)^{-1}\psi_\cyc(c) \in F_{\fkp},
$$
which is clearly independent of the choices of $c$ and the topological generator $u$. It is expected that $\fkL(g)\neq 0$, but in general this is only known 
if $g$ is the newform associated with an elliptic curve.

We may take $c$ to be the image of the class in $H^1(\BQp,\sO_{\fkp}(1))$ of the $G_\BQp$-representation $\CT$ as the latter
can also be realized as a $G_\BQp$-extension 
$$
0\rightarrow \sO_{\fkp}(1) \rightarrow \CT \rightarrow \sO_{\fkp} \rightarrow 0
$$
that yields $\CV$ by extension of scalars.
This choice of the class $c$ is well-defined up to $\sO_{\fkp}^\times$-multiple. The image of $c$ in $H^1(\BQp, k(1))$
is just the class $\ov c$ of the reduction of $\CT$ modulo $\fkp$. That is, $\ov c$ is the class of the $G_\BQp$-representation $V$,
which is well-defined up to $k^\times$-multiple.  Since $V$ is assumed to be irreducible, these classes are independent 
of the isogeny class of the abelian variety $A$, up to the indicated multiples. 
Replacing $F_{\fkp}$ with $k$ in the definition of $\psi_\ur$ and $\psi_\cyc$ yields a $k$-basis $\{\ov\psi_\ur,\ov\psi_\cyc\}$
of $H^1(\BQp,k)$ such that, for the choices of $c$ and $\ov c$ in the previous paragraph
$$
\ov\psi_\ur(\ov c) = \psi_\ur(c) \pmod\fkp \ \ \text{and} \ \ \ov\psi_\cyc(\ov c) = \psi_\cyc(c) \pmod\fkp.
$$

Recall that $V$ is said to be finite at $p$ if $V$ arises as the $G_\BQp$-representation on the $\ov\BQ_p$-points of a finite flat group
scheme over $\BZp$. Just as $\psi_\ur(c) \neq 0$ if and only if $\CV$ is not crystalline as a $G_\BQp$-representation, 
$\ov\psi_\ur(\ov c))\neq 0$ if and only if $V$ is not finite at $p$ (see \cite[Prop.~8.2]{Edix-SerreWeight} and
\cite[(2.4.7)]{Serre-Duke}). 

The following lemma will help us get around the problem of the possible vanishing of the $\fkL$-invariant for a general $g$.

\begin{lem}\label{Linv lem}
If $V$ is not finite at $p$, then $\ord_\fkp(\fkL(\CV)) = \ord_\fkp(p)$ if and only if $\ov\psi_\cyc(\ov c)\neq~0$.
\end{lem}
\noindent In particular, if $V$ is not finite at $p$ and $\ov\psi_\cyc(\ov c)\neq 0$, then $\fkL(V)\neq 0$.

\begin{proof} Since $V$ is not finite at $p$, $\ov\psi_\ur(\ov c) \neq 0$ and so $\psi_\ur(c)\in \sO_{\fkp}^\times$. 
Therefore, in this case, $\ord_\fkp(\fkL(\CV)) = \ord_\fkp(\log_p u) + \ord_\fkp(\psi_\cyc(c))$. As $\log_p u\in p\BZ_p^\times$,
if follows that 
$\ord_\fkp(\fkL(\CV)=\ord_\fkp(p)$ if and only if $\psi_\cyc(c)\in\sO_{\fkp}^\times$, which holds
if and only if $\ov\psi_\cyc(\ov c) \neq 0$.
\end{proof}

For ease of later reference we consider the following hypotheses for a pair $(g,\fkp)$:
\medskip

\noindent{\it Hypothesis $\fkL$}
\begin{itemize}
\item If $p\mid\mid N$ and $a(p)=1$, then $\ov\psi_\cyc(\ov c) \neq 0$.
\end{itemize}

\noindent Clearly, this is a hypothesis only on the residual representation $V$ (even on $\ov\rho$ if $V$ is irreducible). 

\begin{remark} If $A$ is an elliptic curve with split multiplicative reduction at $p$, then the parameterization in the proof of
Lemma \ref{ord res lem} is just the Tate parameterization:  $X=Y$ is a free $\BZ$-module of rank one and the image of $j:X\rightarrow 
T(\BQp)\cong \BQ_p^\times$ is $q_A^\BZ$ for some $q_A\in \BQ_p^\times$ with $\ord_p (q_A) > 0$; this is the so-called Tate period of
$A$. It then follows from the definitions that the $\fkL$-invariant in this case is just $\fkL(\CV) = \fkL(V_pA) = \log_p q_A/\ord_p (q_A)$.
Since $A$ does not have complex multiplication, 
$q_A$ is transcendental 
by a theorem of Barr\'e-Sirieix, Diaz, Gramain, and Philibert \cite{BDGP}, and so $\log_p q_A\neq 0$.
 In particular, the $\fkL$-invariant is non-zero in this case. This non-vanishing is not known in general 
for an arbitrary $g$ with split multiplicative reduction. The purpose of Lemma \ref{Linv lem} is to give conditions that ensure the
non-vanishing of the $\fkL$-invariant for a general $g$ and that continue to hold for suitable newforms congruent to~$g$.
\end{remark}

For ease of later use we also note that in the case that $A$ is an elliptic curve with split multiplicative reduction, the conditions in 
Lemma \ref{Linv lem} can be rewritten in terms of the Tate period $q_A$. 

\begin{lem}\label{Linv lem EC}
Suppose $A$ is an elliptic curve with split multiplicative reduction at $p$ and let $q_A\in\BQ_p^\times$ be its Tate period. Then
\begin{itemize}
\item[\rm (i)] $A[p]$ is not finite at $p$ if and only if $p\nmid\ord_p(q_A)$,
\item[\rm (ii)] $\ov\psi_\cyc(\ov c)\neq 0$ if and only if $\ord_p(\log_p q_A) = 1$  (that is, $\log_p q_A \in p\BZ_p^\times$).
\end{itemize}
\end{lem}
\noindent In part (ii), $\ov c$ is the class in $H^1(\BQ_p,\BF_p(1))$ associated to $A[p]$ as in Lemma \ref{Linv lem}.

\begin{proof}
Part (i) follows from the Tate parameterization $A(\ov\BQ_p) \cong \ov\BQ_p^\times/q_A^\BZ$:
$\ov c$ is the image of $q_A$ in $\BQ_p^\times/(\BQ_p^\times)^p \isoarrow H^1(\BQ_p,\BF_p(1))$, and this belongs to the
image of $\BZ_p^\times/(\BZ_p^\times)^p$ (that is, $A[p]$ is finite at $p$) if and only if $p\mid \ord_p(q_A)$. 
Writing $q_A = \omega\cdot u^a \cdot p^t$ with $\omega\in \mu_{p-1}$ and $a\in \BZ_p$ (recall that $u$ is a topological
generator of $1+p\BZ_p$), we see from the definition of $\ov\psi_\cyc$ that $\ov\psi_\cyc(\ov c)\neq 0$ if and only if $p\nmid a$,
that is, if and only if $\ord_p(\log_p q_A) = \ord_p(pa) = 1$. This proves part (ii).
\end{proof}

\subsection{Convention for the modifier `$g$'}
If it is necessary to distinguish some of the objects associated with a particular newform $g$ (e.g., $a(n)$, $\sO$, $A$, $V_0$, etc.) we will indicate
them by a subscript `$g$' (e.g, $a_g(n)$, $\sO_g$, $A_g$, $V_{g,0}$, etc.).

\subsection{Kolyvagin primes and the set $\Lambda$}

Let $g$ and $\fkp$ be as in \ref{newform-gal}.
A prime $\ell\nmid NDp$ is a called a {\it Kolyvagin prime} (with respect to $g$ and $\fkp$) if $\ell$ is inert in $K$ and the Kolyvagin index
$$
M(\ell) = \min\{\ord_\fkp(\ell+1),\ord_\fkp(a(\ell))\}
$$ 
is positive. We let $\Lambda$ be the set of squarefree products $n$ of such
Kolyvagin primes, and for $n\in\Lambda$ we put
$$
M(n) = \min\{M(\ell) \ : \ \ell\mid n\}.
$$

\subsection{Admissible primes and the set $\Lambda'$}

Let $g$ and $\fkp$ be as in \ref{newform-gal}.
A prime $q\nmid NDp$ is called {\em admissible} (with respect to $g$ and $\fkp$) if $q$ is inert in $K$, $p\nmid (q^2-1)$, and $\ord_\fkp((q+1)^2-a(q)^2)\geq 1$.
We let $\Lambda'$ be the set of squarefree products $m$ of such admissible primes, and $\Lambda^{',\pm}\subset\Lambda'$
the subset of $m$ such that $(-1)^{\nu(m)}=\pm 1$, where $\nu(m)$ is the number of prime factors of $m$.

\subsection{Permissible factorizations}
Given  a positive integer $M$, a factorization $M=M^+M^-$ is {\it permissible} (with respect to $K$) if $M^+$ and $M^-$ are coprime positive integers,
$M^-$ is square-free, $M^+$ is divisible only by primes that split in $K$, and $M^-$ is divisible only by primes that are inert in $K$.
Note that given $K$ and $M$, a permissible factorization need not exist, but if one exists then it is, of course, unique.

\subsection{Hypothesis $\heart$ for $(g,\fkp,K)$}

Let $g$ and $\fkp$ be as in \ref{newform-gal}. Let $\Ram(\ov{\rho})$ be the set of all primes $\ell\mid\mid N$, $\ell\neq p$, such 
that $\ov\rho$ is ramified at $\ell$.
We consider the following hypothesis for $(g,\fkp,K)$: 
\medskip

\noindent {\em Hypothesis $\heart$}
\begin{itemize}
\item[(1)] A permissible factorization $N=N^+N^-$ exists 
($N^-=1$ is allowed).
\item[(2)] 
 $\Ram(\ov{\rho}_{})$ contains all primes $\ell\neq p$ such that $\ell\mid\mid N^+$ and all primes $\ell\mid N^-$ such that $\ell\equiv \pm 1\pmod p$.
  \item[(3)] $\Ram(\ov{\rho})\neq\emptyset$, and either $\Ram(\ov{\rho})$ contains a prime $\ell\mid N^-$ or there are at least two primes $\ell\mid\mid N^+$.
 \item[(4)] For all primes $\ell$ with $\ell^2\mid N^+$,  $H^1(\BQ_\ell, V)=0$ (equivalently, $V^{G_{\BQ_\ell}}=0$). 
\end{itemize}

\begin{remark}
Part (3) implies that $\ov\rho$ is ramified at some prime $\ell\neq p$ such that $\ell\mid\mid N$. For such an $\ell$, the image of $I_\ell$
under $\ov\rho$ is unipotent. In particular, since $\ov\rho$ is semisimple it must be that $\ov\rho$ is irreducible.
That is, implicit in Hypothesis $\heart$ is the irreducibility of $\ov\rho$ (and hence of $V$).
\end{remark}

\begin{remark}
If $A$ is an elliptic curve and $p\geq 5$, then $(4)$ is always satisfied (see \cite[Lem.~5.1(2)]{Z13}). So in this case Hypothesis $\heart$ is just Hypothesis $\spade$ from
the Introduction.
\end{remark}

\section{Shimura Curves and Heegner Points}

Let $N$ be a positive integer and suppose $N=N^+N^-$ is factorization with
$N^+$ and $N^-$ coprime positive integers and $N^-$ square-free.

\subsection{Shimura curves and Shimura sets}

If $N^-$ is a product of an even number of primes ($N^-=1$ is allowed), 
let $B=B_{N^-}$ be the indefinite quaternion algebra of discriminant $N^-$ and $R \subset B$ a fixed 
Eichler order of level $N^+$.  We then let $X_{N^+,N^-}$ be the associated Shimura curve. 
This curve has a canonical model over $\BQ$ with complex parameterization:
$$
X_{N^+,N^-}(\BC) = B^\times\bs [\mathfrak{h}^{\pm} \times \wh B^\times/ \wh R^\times],
$$
where $\wh B = B\otimes\wh\BZ$ and $\wh R = R\otimes\wh\BZ$. The action of $B^\times$ on $\mathfrak{h}^\pm=\BC\backslash\BR$ is via an isomorphism $B\otimes\BR \cong  M_2(\BR)$ and the usual action of $\GL_2(\BR)$ on $\mathfrak{h}^\pm$.

If $N^-$ is a product of an odd number of primes, the role of the Shimura curve (which does not exist) is frequently played by the Shimura
set $X_{N^+,N^-}$ determined by taking $B$ to be the definite quaternion algebra $B=B_{N^-\infty}$ of discriminant $N^-\infty$
and $R\subset B$ an Eichler order of level $N^+$ as before:
$$
X_{N^+,N^-} = B^\times \bs \wh B^\times/ \wh R^\times.
$$
This is a finite set. It classifies locally-free left $R$-modules of rank one.

Frequently, when describing certain constructions we will assume that we have fixed identifications
$\wh B^{N^-}  = M_2(\BA_f^{N^-})$ such that $\wh R^{N^-}$ is identified with the subring of
$M_2(\wh\BZ^{N^-})$ with lower left entry a multiple of $N^+$. Here the superscript `$N^-$' denotes the 
finite adeles away from the primes dividing $N^-$.

\subsection{Some Hecke rings and Hecke actions}

Let $S(N^+,N^-)$ be the $N^-$-new subspace of $S_2(\Gamma_0(N))$. 
Let $\BT_{N^+,N^-}$  be the usual Hecke ring acting faithfully on $S(N^+,N^-)$ and let $\BT_0(N^+,N^-)$ be its $p$-adic completion. 
These are generated by the Hecke operators $T_\ell$, for $\ell\nmid N$, and $U_\ell$, for $\ell\mid N$. Note that
each $U_\ell$, $\ell\mid N^-$, acts as an involution (the eigenvalues of such a $U_\ell$ are $\pm 1$ since 
$\ell\mid\mid N$ and the nebentypus is trivial).

Suppose $N^-$ is a product of an even number of primes. 
There is a natural action of $\BT_{N^+,N^-}$ on the Jacobian 
$$
J(X_{N^+,N^-})  = \Pic^0(X_{N^+,N^-}).
$$
which gives an inclusion $\BT_0(N^+,N^-)\hookrightarrow\End(J(X_{N^+,N^-}))\otimes\BZp$.
In fact, the actions of the $T_\ell$ and $U_\ell$ can be defined through correspondences exactly as in the $N^-=1$ case.

If $N^-$ is a product of an odd number of primes, then there is an action of $\BT_{N^+,N^-}$ on the $0$-divisors 
$$
\CS_{N^+,N^-} = \{\sum a_x\cdot x \in \BZ[X_m] \ : \ \sum a_x = 0\}
$$ 
of the Shimura set $X_{N^+,N^-}$. This gives a homomorphism $\BT_0(N^+,N^-)\hookrightarrow\End_\BZ(\CS_{N^+,N^-})\otimes\BZ_p$.

Both these actions reflect the Jacquet-Langlands correspondence, which gives a Hecke-equivariant isomorphism between 
$S(N^+,N^-)$ and the space of weight $2$ cuspforms of level $\wh R^\times$ for the multiplicative group $B^\times$ 
(when $N^-$ is a product of an odd number of primes, this space of cuspforms is naturally identified with $\CS_{N^+,N^-}\otimes\BC$).

\subsection{Heegner points}\label{Heegner pts}

In this section we follow \cite[\S2.2-2.3]{Z13}, where more details can be found. We
assume in this section that $N=N^+N^-$ is a permissible factorization.

Suppose $N^-$ is a product of an even number of primes. Let $A$ be an abelian variety quotient of $J(X_{N^+,N^-})$. 
For each integer $n\in\Lambda$ the theory of Heegner points yields Heegner points of level $n$ 
$$
x_{N^+,N^-}(n)\in X_{N^+,N^-}(K[n]) \ \ \text{and} \ \  
y_A(n)\in A(K[n])
$$
defined over the ring class field extension $K[n]/K$  of conductor $n$. 
When $N^-\neq 1$, there is a choice of an auxiliary prime $\ell_0\nmid NDpn$ that intervenes in the 
definition of $y_A(n)$ (which is the image of a point in $J(X_{N^+,N^-})$ that is determined by $x_{N^+,N^-}(n)$ and $\ell_0$). 
The point $x_{N^+,N^-}(n)$ can be described in terms of the complex parameterization: $x_{N^+,N^-}(n)$ is the
double coset $[h_0\times h]$ represented by $h_0\times h \in \fkh^\pm\times\wh B^\times$, where
$h_0$ is the unique fixed point of the action of $K^\times$ on $\fkh=\fkh^+$ for
$K\hookrightarrow B$ an (optimal) embedding such that $K\cap R = \sO_K$ and
$h=(h_\ell)$ with $h_\ell = \diag(\ell,1)$ if $\ell|n$ and $h_\ell =1$ otherwise.

Suppose $N^-$ has an odd number of factors. There is also a Heegner point
$x_{N^+,N^-}(n) = [h]$ in the Shimura set $X_{N^+,N^-}$, which is represented by $h\in\wh B^\times$ as above. 

\subsection{The points $y_{A,K}$ and $x_{N^+,N^-,K}$}
If $N^-$ is a product of an even number of primes and $A$ is an abelian variety quotient of $J(X_{N^+,N^-})$, then we let
$$
y_{A,K} = \tr_{K[1]/K}( y_A(1) )= \sum_{\sigma\in \Gal(K[1]/K)} \sigma (y_A(1)) \in A(K).
$$
If $N^-$ is a product of an odd number of primes, then we let
$$
x_{N^+,N^-,K} = \tr_{K[1]/K}(x_{N^+,N^-}(1)) = \sum_{\sigma\in\Gal(K[1]/K)} \sigma(x_{N^+,N^-}(1)) \in \BZ[X_{N^+,N^-}].
$$
For this last, the action of $\Gal(K[1]/K)$ on $X_{N^+,N^-}$ is via the reciprocity law:
$$
\mathrm{rec}: \Gal(K[1]/K)\isoarrow K^\times\backslash \widehat K^\times/\widehat \sO_{K}^\times,
$$ 
and
$$
\sigma([h]) = [\mathrm{rec}(\sigma)h].
$$

\subsection{Reduction of $X_{N^+,N^-}$ at $p$ when $p\mid\mid N^+$}\label{red at p}
Suppose $N^-$ is a product of an even number of primes and $p\mid\mid N^+$. 
We recall some properties of the reduction of the Shimura curves $X_{N^+,N^-}$ at the prime $p$. 

As explained in \cite[\S10]{He}, $X_{N^+,N^-}$ has a regular model over $\BZ_{(p)}$ that is 
a course moduli space for false elliptic curves with level structure. This model is smooth away from the supersingular points
on the special fibre, and the special fibre can be identified with two copies of the Shimura curve for the Eichler order of level $N^+/p$
that are glued transversely at the supersingular points. The completion of the strict Henselization of the local ring of $X_{N^+,N^-}$ 
at a supersingular point is isomorphic to  $W(\ov\BF_p)[\![x,y]\!]/(xy-p)$.  
The supersingular points $X_{N^+,N^-}(\ov\BF_p)^{ss}$ are all defined over $\BF_{p^2}$ and can be naturally identified with the Shimura set 
$X_{N^+/p,pN^-}$ for the definite
quaternion algebra of discriminant $\infty p N^-$ and an Eichler order of level $N^+/p$ (cf.~\cite[\S5,6]{Milne1979}).

The N\'eron model of the Jacobian $J(X_{N^+,N^-})/\BQp$ has semistable reduction: the connected component of the 
special fibre containing the identity element is the extension of an abelian variety (the product of two copies of the 
reduction of $J(X_{N^+/p,N^-})$) by a torus. Let $\CX_{N^+,N^-}$ be the character group of this torus. Then $\CX_{N^+,N^-}=H_1(\CG,\BZ)$, 
where $\CG$ is the dual graph of the special fibre of $X_{N^+,N^-}$. 
This all follows from the existence of the model described above and \cite[Prop.~9]{MR}.
The set of vertices of $\CG$ is just the set of irreducible components
of $X_{N^+,N^-}/\BF_p$ (so there are two vertices), the edges of $\CG$ connecting two vertices are just the set of singular points in the intersection
of the two components (so in this case the edges are just $X_{N^+,N^-}(\ov\BF_p)^{ss} = X_{N^+/p,N^-p}$), and 
$$
H_1(\CG,\BZ) = \CS_{N^+/p,N^-p}.
$$

The maps defining the correspondences giving the action of the Hecke operators can be described in 
terms of the moduli problem underlying the model of $X_{N^+,N^-}$, and so  
define Hecke operators 
on $\CX_{N^+,N^-}$. 
In particular, there is an action of
$\BT_{N^+,N^-}$ on $\CX_{N^+,N^-}$ that is compatible with the Hecke actions on $\CS_{N^+/p,N^-p}$;
this action goes via the projection $\BT_{N^+,N^-} \twoheadrightarrow \BT_{N^+/p,N^-p}$ onto the $p$-new Hecke algebra.

\subsection{Reduction of $X_{N^+,N^-}$ at a prime $q\neq p$}
Suppose $N^-$ is a product of an even number of primes.
Let $q$ be a prime such that $q\nmid N^+$. 
The reduction of $X_{N^+,N^-}$ at $q$ is described in \cite[\S2.4]{Z13}. We recall some of this here.

If $q\nmid N$, then $X_{N^+,N^-}$ has a smooth model over $\BZ_{(q)}$. The supersingular points on the special fibre are defined over
$\BF_{q^2}$ and $X_{N^+,N^-}(\BF_{q^2})^{ss}$ is naturally identified with the Shimura set $X_{N^+,N^-q}$.  

If $q\mid N^-$, then $X_{N^+,N^-}$ has a minimal regular model over $\BZ_{(q)}$ that comes equipped with a $q$-adic (Cerednik-Drinfeld) uniformization (cf.~\cite[\S4]{R90}). The special fibre is a union of smooth curves intersecting transversely at the singular points.  
The set $\sV(X_{N^+,N^-})$ of irreducible components of the special fibre is identified with two copies of the Shimura set $X_{N^+,N^-/q}$:
$$
\sV(X_{N^+,N^-}) = X_{N^+,N^-/q}\times \BZ/2\BZ.
$$ 
Let $\sV(X_{N^+,N^-})_0 = X_{N^+,N^-/q}\times\{0\} \subset \sV(X_{N^+,N^-})$.

\subsection{Reduction modulo $q$ of Heegner points}\label{Heegner red}
Let $q$ be a prime that is inert in $K$. Again we suppose $N=N^+N^-$ is a permissible factorization and also that 
$N^-$ is a product of an even number of primes. We recall some facts about 
the reduction modulo $q$ of the Heegner points $x_{N^+,N^-}(n)$ from \ref{Heegner pts}. More details can be found in \cite[\S2.5]{Z13}.

If $q\nmid N$, then, since $q$ is inert in $K$, the reduction $\Red_q(x_{N^+,N^-}(n))$ of each Heegner point $x_{N^+,N^-}(n)$ is a supersingular point in the special fibre.  
If $x_{N^+,N^-}(n) = [h_0\times h]$ as in \ref{Heegner pts}, then $\Red_q(x_{N^+,N^-}(n))$ is the point $[h]\in X_{N^+,N^-q} = X_{N^+,N^-}(\ov\BF_q)^{ss}$.

If $q\mid N^-$, then the reduction modulo $q$ of $x_{N^+,N^-}(n)$ lies on a component $\Sp_q(x_{N^+,N^-}(n))\in \sV(X_{N^+,N^-})_0$.
If $x_{N^+,N^-}(n) = [h_0\times h]$, then $\Sp_q(x_{N^+,N^-}(n)) = [h]\times 0 \in X_{N^+,N^-/q}\times\{0\} = \sV(X_{N^+,N^-})_0$.

These facts are essentially summarized in the following lemma, which recalls Theorem \cite[Thm.~2.1]{Z13}. 

\begin{lem}$(\cite[Thm.~2.1]{Z13})$ Suppose $N^-$ is a product of an even number of primes. Let $q$ be a prime that is inert in $K$. 
\begin{itemize}
\item[\rm (i)] If $q\nmid N$, then $\Red_q(x_{N^+,N^-}(n)) = x_{N^+,N^-q}(n) \in X_{N^+,N^-q}$.
\item[\rm (ii)] If $q\mid N^-$, then $\Sp_q(x_{N^+,N^-}(n)) = x_{N^+,N^-/q}(n)\in X_{N^+,N^-/q}$.
\end{itemize}
\end{lem}
\noindent Since $q\neq p$ and this lemma is only about reductions modulo $q$, it also holds when $p\mid N$.

\section{Kolyvagin's Conjecture}

We recall Kolyvagin's Conjecture for a newform.
Let $g$ be a newform of weight 2, level $N$, and trivial nebentypus, and let $\fkp\subset \sO$ be a prime as in \ref{newform-gal}.
Suppose $N=N^+N^-$ is a permissible factorization with $N^-$ a product of an even number of primes.

\subsection{Optimal quotients}\label{optimal qts}
Let $I\subset \BT_{N^+,N^-}$ be the kernel of the homomorphism $\pi:\BT_{N^+,N^-}\twoheadrightarrow \sO_{0}$ sending 
$T_\ell$ or $U_\ell$ to $a(\ell)$.
Let $A_{0} = J(X_{N^+,N^-})/IJ(X_{N^+,N^-})$. Then $A_{0}$ together
with the projection map $J(X_{N^+,N^-})\twoheadrightarrow A_{0}$ is the optimal quotient associated with $g$ (in particular, the kernel of
the projection is connected). 
There is clearly an induced embedding $\sO_{0}\hookrightarrow \End_\BQ A_{0}$,
but this does not necessarily extend to an action of $\sO$.

We can and do assume that $A$ is chosen so that there is a quotient map $J(X_{N^+,N^-})\twoheadrightarrow A$ that factors as the composition
of the optimal quotient with an isogeny $A_{0}\rightarrow A$ and that the corresponding image of $\Ta_p J(X_{N^+,N^-})$ in $\Ta_p A$ is not
contained in $\fkp\Ta_pA$.  Then $A$ together with the projection $J(X_{N^+,N^-})\twoheadrightarrow A$ is an 
optimal $(\sO,\fkp)$-quotient in the sense of \cite[\S2.7]{Z13}.

\subsection{The Conjecture}\label{Koly Conj}
As explained in \cite[\S2.7]{Z13}, by applying Kolyvagin's derivative operators to the points $y(n) = y_A(n)\in A(K[n])$ from \ref{Heegner pts}, 
for each $n\in\Lambda$ and each non-negative integer $0\leq  M\leq M(n)$ one obtains cohomology classes 
$$
c_M(n) \in H^1(K,A_{M}) , \ \ \ A_{M} = \Ta_pA\otimes_{\sO}\sO/\fkp^M \cong A[\fkp^M].
$$
Furthermore, letting $\eps_n = \eps\cdot(-1)^{\nu(n)} \in \{\pm 1\}$, where $\eps\in\{\pm 1\}$ is the root number of $g$, we have 
$$
c_M(n) \in H^1(K,A_{M})^{\eps_n},
$$
where the superscript denotes the subspace where the non-trivial automorphism of $K$ acts as multiplication by $\eps_n$.

Let 
$$
\sM(n) = \max\{M\geq 0 \ : \  c_M(n) \in \fkp^{M'}H^1(K,A_{M'}) \ \forall M'\leq M\}
$$ 
and
$$
\sM_r = \min\{\sM(n) \ : \ \text{$n\in \Lambda$ has exactly $r$ prime factors}\}.
$$
We allow $\infty$ as a value for $\sM(n)$ and $\sM_r$.
Kolyvagin showed that $\sM_r\geq \sM_{r+1}\geq 0$. Let
$$
\sM_\infty(g) = \min\{\sM_r \ : \ r\geq 0\}.
$$

\begin{conjecture}$(\text{Kolyvagin's Conjecture})$
Assume that Hypothesis $\club$ holds for $\ov\rho$.
Then the collection of cohomology classes
$$
\kappa^\infty = \{c_M(n) \ : \ n\in\Lambda, M\leq M(n)\}
$$
is non-zero. Equivalently, $\sM_\infty(g) <\infty$.
\end{conjecture}

\noindent  In \cite{Z13} this conjecture was proved under certain hypotheses on $g$ and $\ov\rho$, including $p\nmid N$.
In this paper we extend this result to certain $g$ and $\ov\rho$ with $p\mid\mid N$ (see Theorem \ref{Kconj thm}).

\section{Level-raising of modular forms}

This section contains the bulk of the new results\footnote{The main results of \cite{Smult} are also needed, replacing the references
to \cite{SU} in \cite[\S7]{Z13}. This is explained in \S\ref{BSD0} below.}
required to extend the methods of \cite{Z13} to certain cases where $p\mid\mid N$. In particular, we prove a simple version of Ihara's lemma for 
Shimura curves that does not seem to be contained in the current literature, as well as a level-raising result\footnote{In addition to being used in this paper,
this permits the arguments in \cite{He} to be extended to more maximal ideals, including those corresponding to the $\ov\rho$
considered herein.}, and multiplicity one results for certain Hecke modules.

\subsection{Ihara's lemma}\label{Ihara} 
Let $N$ be a positive integer and $N=N^+N^-$ a factorization such that $N^+$ and $N^-$ are coprime, $p\mid\mid N^+$, and $N^-$ is a square-free 
product of an even number of primes. This need not be a permissible factorization with respect to $K$.

Let $\fkm\subset \BT = \BT_0(N^+,N^-)$ be a maximal ideal and let $k_\fkm = \BT/\fkm$. Associated with $\fkm$ is a 
semisimple two-dimensional $k_\fkm$-representation of $G_\BQ$
$$
\ov\rho_\fkm: G_\BQ\rightarrow \Aut_{k_\fkm} V_\fkm
$$
that is unramified at the primes $\ell\nmid N$ (since $p|N$) and, for such primes $\ell$, satisfies 
$\Trace\ov\rho_\fkm(\Frob_\ell) = T_\ell \pmod{\fkm}$.
For example, if $\fkm$ is the kernel of the reduction modulo $\fkp$ of the map $\BT\rightarrow \sO_{\fkp}$
associated with a newform $g$ as in \ref{newform-gal}, then $V_\fkm \cong V_{0}$.

Let $q\nmid N$ be a prime. Let $\BT_1 = \BT = \BT_0(N^+,N^-)$ and let $\BT_2 = \BT_0(N^+q,N^-)$.
Let $\BT_i^{\{q\}}\subset \BT_i$ be the subalgebra\footnote{
The argument used to prove claim 1 in the proof of the lemma on p.~491 of \cite{WilesFLT} shows that $\BT_1^{\{q\}} = \BT_1$.}
generated by omitting the Hecke operator $T_q$ or $U_q$.  There is a surjective homomorphism 
$\BT_2^{\{q\}} \rightarrow \BT_1^{\{q\}}$ that sends $T_\ell$ and $U_\ell$, respectively, to $T_\ell$ and $U_\ell$ for all $\ell\neq q$.
Let $\fkm_q = \fkm \cap \BT_1^{\{q\}}$. We also write $\fkm_q$ for the maximal ideal of $\BT_2^{\{q\}}$ that is the preimage
of $\fkm_q$.

In the following we prove two versions of Ihara's lemma, one each for a definite and an indefinite case, 
with the former used to prove the latter. These are straightforward, but the indefinite version fills part of an apparent gap in the 
current literature on Ihara's lemma for Shimura curves.  

\paragraph{The definite case.}  
Let  $\CS_1
= \CS_{N^+/p,N^-p}\otimes\BQp/\BZp$.
As recalled in \ref{red at p}, this has an action of $\BT_1 = \BT_0(N^+,N^-)$ through its $p$-new quotient $\BT_0(N^+/p,N^-p)$.
Similarly, let $\CS_2 = \CS_{N^+q/p,N^-p}\otimes\BQp/\BZp$; this has an action of $\BT_2$.

There are two degeneracy maps $\alpha,\beta: X_{N^+q/p,N^-p}\rightarrow X_{N^+/p,N^-p}$ given, respectively, by 
$g\mapsto g$ and $g\mapsto gd_q^{-1}$, $d_q = \diag(q,1)\in \GL_2(\BQ_q)$.
These induce homomorphisms $\alpha^*,\beta^*:\CS_1\rightarrow \CS_2$ that commute with the actions
of $\BT_1^{\{q\}}$ and $\BT_2^{\{q\}}$. 

\begin{lem}\label{Ihara-def} Suppose $\ov\rho_\fkm$ is irreducible. Suppose also that $\fkm$ is new at $p$ in the sense
that it is the preimage of a maximal ideal of the $p$-new Hecke ring $\BT_0(N^+/p,N^-p)$.
The homomorphism
$$
\CS_1[\fkm_q^\infty]\times \CS_1[\fkm_q^\infty]\stackrel{\alpha^*+\beta^*}{\longrightarrow} \CS_2[\fkm_q^\infty]
$$
is an injection.
\end{lem}

\begin{proof} The same argument used to prove \cite[Lem.~2]{DiamondTaylor} shows that the kernel is annihilated by $T_\ell  - 1 -\ell$
for all $\ell\equiv 1 \pmod{Nq}$. But it cannot be that $T_\ell-1-\ell\in \fkm_q$ for all such $\ell$, for then we would have 
$\Trace\ov\rho_\fkm(\Frob_\ell) = 1+\ell \pmod{\fkm}$ for all $\ell\equiv 1 \pmod{Nq}$, which would imply that $\ov\rho_{\fkm}$ is reducible.
\end{proof}

\paragraph{The indefinite case.}
There are two degeneracy maps $\alpha,\beta:X_{N^+q,N^-}\rightarrow X_{N^+,N^-}$ over $\BZ_{(p)}$.
In terms of the complex parameterizations, these correspond to $\tau\mapsto \tau$ and 
$\tau\mapsto q\tau$, respectively.
These maps induce the above similarly-denoted degeneracy maps on the supersingular points of the special fibres upon fixing
compatible identifications of the sets of supersingular points with the Shimura sets $X_{N^+q/p,N^-p}$ and $X_{N^+/p,N^-p}$ as before.
Let $J_1 = J(X_{N^+,N^-})$ and $J_2 = J(X_{N^+q, N^-})$. 
The maps $\alpha$, $\beta$ induce homomorphisms $\alpha^*,\beta^*:J_1\rightarrow J_2$ by 
Picard functoriality. These maps
are compatible with the actions of $\BT_1^{\{q\}}$ and $\BT_2^{\{q\}}$. 

The version of Ihara's Lemma that we will need for the indefinite case is:

\begin{lem}\label{Ihara with p} If
\begin{itemize}
\item[\rm (a)] $\ov\rho_{\fkm}$ is irreducible, and
\item[\rm (b)] $\ov\rho_{\fkm}$ is not finite at $p$,
\end{itemize}
then the morphism
$$
J_1[\fkm_q^\infty]\times J_1[\fkm_q^\infty]\stackrel{\alpha^* + \beta^*}{\longrightarrow} J_2[\fkm_q^\infty],
$$
is injective.
\end{lem}

\noindent Since $\BT_i$ acts on $J_i[p^n] = J_i(\ov\BQ)[p^n] = J_i(\ov\BQ_p)[p^n]$, the $\fkm_q^n$-torsion of $J_i$ is well-defined.

\begin{proof} 
We note that since $\ov\rho_\fkm$ is not finite at $p$ by (b), $\fkm$ is $p$-new: $\fkm$ is the preimage of a maximal ideal
of $\BT_0(N^+/p,N^-p)$. 

Let $J_i^f$ be the maximal $p$-divisible subgroup of $J_i/\BQp$ that extends to a $p$-divisible subgroup over $\BZp$,
and let $J_i^t\subset J_i^f$ be the maximal $p$-divisible subgroup that extends to the $p$-divisible subgroup of a torus over $\BZp$.
The character group of this torus is canonically identified with the character group $\CX_i$ 
of the toric part of the special fibre of the N\'eron model of $J_i$ over $\BZp$, even as $G_{\BQp}$-modules. 
As we have explained in \ref{red at p}, $\CX_i$ is identified with $\CS_{N^+_i/p,N^-p}$ (where $N^+_1 = N^+/p$ 
and $N_2^+ = N^+q/p$), even as Hecke modules.
Furthermore, the Weil-pairing $J_i[p^n]\times J_i[p^n]\rightarrow \mu_{p^n}$
induces an identification $J_i[p^n]/J_i^f[p^n] = \Hom_{\BZp\mathrm{-mod}}(J_i^t[p^n],\mu_{p^n}) = \CX_i/p^n \CX_i$ 
(as both Hecke and $G_\BQp$-modules),
so $J_i[p^\infty]/J_i^f[p^\infty] =  \CS_i$.  Therefore we have exact sequences 
$$
0\rightarrow J_i^f[\fkm_q^\infty] \rightarrow J_i[\fkm_q^\infty] \rightarrow \CS_i[\fkm_q^\infty]\rightarrow 0
$$
of $\BT_i^{\{q\}}$-modules.

The maps $\alpha^*$ and $\beta^*$ induce maps between the above exact sequences for $i=1,2$, 
and we have a commutative diagram:
$$
\begin{tikzcd}
0 \arrow{r}  & J_1^f[\fkm_q^\infty]\times J_1^f[\fkm_q^\infty] \arrow{r}\arrow{d}{\alpha^*+\beta^*} &  J_1[\fkm_q^\infty]\times
J_1[\fkm_q^\infty]\arrow{r}\arrow{d}{\alpha^*+\beta^*}
& \CS_1[\fkm_q^\infty]\times\CS_1[\fkm_q^\infty]\arrow{r}\arrow{d}{\alpha^*+\beta^*} & 0 \\
0 \arrow{r}  & J_2^f[\fkm_q^\infty] \arrow{r} &  J_{Nq,M}[\fkm_q^\infty]\arrow{r}
& \CS_2[\fkm_q^\infty]\arrow{r}& 0.
\end{tikzcd}
$$
By Lemma \ref{Ihara-def}, the right vertical map is an injection. So the kernel of the middle vertical map equals the kernel of the left vertical 
map. But if the kernel of the middle map is non-zero, then its $\fkm_q$-torsion, being a $G_\BQ$-stable submodule of
$J_1[\fkm_q]\times J_1[\fkm_q]$, must contain a submodule isomorphic to $V_\fkm$ (as $J_1[\fkm_q]$ is a sum of copies of $V_\fkm$ 
by (a) and \cite[Thms.~1 and 2]{BosLenRib}). But this would imply
that $V_\fkm$ is a $G_\BQ$-submodule of $J_{1}^f[\fkm_q]\times J_{1}^f[\fkm_q]$ and therefore $V_\fkm$ would be 
finite at $p$, contradicting (b).
\end{proof}

To be precise, the extension of the argument of Bertolini and Darmon (from \cite{BD-ACMC}) needed to extend
the proof of \cite[Thm.~4.3]{Z13} to the cases considered in this paper, depends on a version of Ihara's Lemma for Shimura curves with 
$\Gamma_1(p)$-structures, not just 
$\Gamma_0(p)$-structures. However, this is an easy consequence of the preceding lemma, as we now explain.

Let $X_{N^+,N^-}'$ be the Shimura curve over $\BQ$ defined by replacing $R$ with its
suborder $R'\subset R$ consisting of elements with reduction modulo $p$ lying in the subgroup of $\GL_2(\BZ/p\BZ)$ with upper
left entry congruent to $1$ modulo $p$. Let $\BT'=\BT_0'(N^+,N^-)$ be the $p$-adic completion of the Hecke algebra acting on the $N^-$-new subspace
of $S_2(\Gamma_1(p)\cap\Gamma_0(N))$; this includes the diamond operators $\langle d\rangle$ for $(d,N)=1$.
The map $J(X_{N^+,N^-})\rightarrow J(X_{N^+,N^-}')$
induced by Picard functoriality from the natural map $X_{N^+,N^-}'\rightarrow X_{N^+,N^-}$ ($\tau\mapsto \tau$ in terms of the 
complex uniformization) is compatible with the natural homomorphism $\BT'\rightarrow\BT$, which sends each $\langle d\rangle$ to $1$,
and the image of $J(X_{N^+,N^-})$ is just the kernel of the diamond operators. 
Let $\fkm_q'\subset \BT^{',\{q\}}$ be the preimage of $\fkm_q\subset \BT^{\{q\}}$.
The maximal ideal $\fkm'_q$ contains each $\langle d \rangle -1$. The composition 
$$
J(X_{N^+,N^-})\rightarrow J(X_{N^+,N^-}') \rightarrow J(X_{N^+,N^-}),
$$
where the second arrow comes from Albanese functoriality, 
is just multiplication by an integer prime to $p$
(see also the reduction from $J_1(Np)$ to $J_1(N,p)$ in the proof of \cite[Thm.~2.1]{WilesFLT} in the case where
$\Delta_{(p)}$ is trivial). It follows that the first arrow induces an 
isomorphism 
$$
J(X_{N^+,N^-})[\fkm_q] \isoarrow J(X_{N^+,N^-}')[\fkm_q'].
$$
Consequently:

\begin{cor}\label{Ihara gamma1}
The injectivity of the map in 
Lemma \ref{Ihara with p} also holds with $J_1$ and $J_2$ replaced with $J(X_{N^+,N^-}')$ and $J(X_{N^+q,N^-}')$, respectively,
and $\fkm_q$ replaced with $\fkm_q'$.
\end{cor}

\subsection{A level-raising lemma}\label{level-raise}

Let $g$ and $\fkp$ be as in \ref{newform-gal}.
We will need the following result, which will allow us to `raise the level' of $g$ to include an arbitrary product of admissible primes.
A factorization of the level $N$ of $g$ plays no role in this result. 

\begin{lem}\label{raise level lem} 
Suppose  $p\mid\mid N$ and $\ov\rho$ is irreducible and not finite at $p$. Let $m\in\Lambda'$ be a product of admissible primes. 
There is a newform $g'$ of level $Nm$,
weight $2$, and trivial nebentypus, and a prime $\fkp'\subset \sO_{g'}$ containing $p$ such that 
$\fkp'_0 = \fkp'\cap \sO_{g',0}$ satisfies $\sO_{g',0}/\fkp'_0 \cong \sO_{0}/\fkp_0 = k_0$
and 
$$
\ov\rho_{g'} \cong \ov\rho
$$
as $k_0$-representations of $G_\BQ$.
\end{lem}

This lemma is not covered by the main result of \cite{Diamond-Taylor2}, which excludes the not-finite-at-$p$ cases. However, the omission is
essentially because the cases of Ihara's lemma established in \cite{DiamondTaylor} also exclude these cases. In light of Lemmas \ref{Ihara-def} and \ref{Ihara with p} it should be possible to carry over the arguments of \cite{Diamond-Taylor2}. However, to make this precise we would need to establish a version of Lemma \ref{Ihara with p} allowing for level at primes dividing $N^-$ (that is, working with orders that are not maximal at such non-split primes). Lacking complete references for the needed models of Shimura curves (this should not be a serious obstacle, however), we content ourselves with a work-around relying on a result of Gee \cite[Cor.~3.1.7]{Gee-types} and Hida theory.

\begin{proof} Let $\Sigma = \{\ell\mid Nm\}$. For each $\ell\in \Sigma$, we fix a representation $\tau_\ell:I_\ell\rightarrow \GL_2(\ov\BQ_p)$
as follows, where $I_\ell\subset G_{\BQ_\ell}$ is the inertia subgroup. Such a $\tau_\ell$ is often called an inertial type. 
For $\ell\nmid m$, let $\WD_\ell(\rho)$ be the $\ov\BQ_p$-representation of the Weil-Deligne group of $\BQ_\ell$
associated with $\rho|_{G_{\BQ_\ell}}$ and let $\tau_\ell = \WD_\ell(\rho)|_{I_\ell}$. For $\ell=p$,
$\WD_p(\rho)$ was defined by Fontaine using $p$-adic Hodge theory; we follow the conventions of
\cite[\S3.1]{Gee-types} for this case.  For $\ell\mid m$ we let $\tau_\ell$ be the trivial representation.  

For each $\ell\in\Sigma$, we pick an irreducible component 
$\ov R_\ell^{\square,\veps,\tau_\ell}$ of the local deformation ring $R_\ell^{\square,\veps,\tau_\ell}$ as follows (here we are following the notation of \cite[\S3.1]{Gee-types}).
For $\ell\nmid m$ we choose the component that contains $\WD_\ell(\rho)$, and for $\ell\mid m$ we choose the
component that contains $\WD_\ell(\sigma_\ell)$ for $\sigma_\ell$ the special representation of $\GL_2(\BQ_\ell)$ or its unramified quadratic
twist, depending on whether the roots $\alpha_\ell$ and $\beta_\ell$ of $x^2-a_\ell(g)x + \ell$ modulo $\fkp_0$ satisfy 
$\{\alpha_\ell,\beta_\ell\} = \{1,\ell\}$ or $\{\alpha_\ell,\beta_\ell\} = \{-1.- \ell\}$ (by the definition of an admissible prime these are distinct possibilities 
and either one or the other possibility holds). As a consequence of these choices, if $\rho_f$ is the $p$-adic Galois representation 
associated with a newform $f$ of level $N_f$, weight $2$, and trivial nebentypus 
such that $\WD_\ell(\rho_f)$ is a point on $\ov R_\ell^{\square,\veps,\tau_\ell}$ for each $\ell\in\Sigma$, then 
$\ord_\ell(N_f) = \ord_\ell(Nm)$ for $\ell\in\Sigma$, $\ell\neq p$, and $f$ is nearly ordinary at $p$. To see this we note that for 
$\ell\neq p$ we have $\WD_\ell(\rho_f)|_{I_\ell} \cong \tau_\ell = \WD_\ell(\rho)$. 
If $\tau_\ell\neq 1$, then this completely determines $\ord_\ell(N_f)$, which equals the conductor of $\tau_\ell$.
If $\tau_\ell=1$ then our choice of component forces\footnote{An easy analysis of the components
of $R_\ell^{\square,\veps,1}$, $\ell\neq p$, shows that if one characteristic zero point on an irreducible component is 
isomorphic to $\WD_\ell(\sigma_\ell)$ for $\sigma_\ell$ special (or a twist of
special), then any other is either $\WD_\ell(\sigma_\ell)$ or unramified with Frobenius eigenvalues having ratio $\ell^{\pm 1}$. 
That the latter such points cannot come
from modular representations follows from the Ramanujan bounds.}
$\WD_\ell(\rho_f)$ to be either the special representation or
its unramified quadratic twist (whichever $\WD_\ell(\rho)$, $\ell\nmid m$, or $\WD_\ell(\sigma_\ell)$, $\ell\mid m$, is) 
and so its conductor is $\ell$.  Finally, the points
on an irreducible component of $R_p^{\square,\veps,\tau_p}[1/p]$ are either all potentially ordinary or all not
potentially ordinary, so by our choice of $\ov R_p^{\square,\veps,\tau_p}$ (which contains the ordinary 
point $\WD_p(\rho)$), $\rho_f|_{G_\BQp}$ must be potentially ordinary, and hence $f$ must be nearly ordinary.

From \cite[Cor.~3.1.7]{Gee-types}, it follows that there exists a newform $f$ of level $N_f$ divisible only by primes
in $\Sigma$, of weight $2$ and trivial nebentypus, and having associated $p$-adic Galois representation
$\rho_f$ such that $\WD_\ell(\rho_f)$ is a point on $\ov R_\ell^{\square,\veps,\tau_\ell}$ for all $\ell\in\Sigma$ and such that
$\rho_f|_{G_\BQp}$ is potentially Barsotti-Tate.  To see this we need to check that hypothesis (ord) of 
\cite[Prop.~3.1.15]{Gee-types} holds. But this is an easy consequence of Hida theory: since $a(p)=\pm 1$, $g$ is ordinary at $p$ and so 
belongs to a Hida eigenfamily,
and any member of this family of weight $2$ but non-trivial nebentypus at $p$ (there are infinitely many such in 
the family) provides the lift required for hypothesis (ord).  
By the observation in the preceding paragraph, $\ord_\ell(N_f) = \ord_\ell(N)$ for all $\ell\neq p$.
Finally, to get the form $g'$ we essentially reverse the preceding
argument: as $f$ is nearly ordinary, twisting $f$ by a character $\psi$ of $p$-power order and conductor if needed, we may assume that $f$ is ordinary
(but possibly losing the triviality of the nebentypus at $p$),
and then $g'$ can be taken to be the member of the Hida eigenfamily containing $f$ such that $g'$ has weight 2 and
trivial nebentypus at $p$ (such a $g'$ will always exist since the nebentypus of $f$ is trivial mod $p$ and away from $p$). 
Note that since $\bar\rho_{g'}
\cong \bar\rho$ is not finite at $p$, $g'$ must be new at $p$, and so the level of $g'$ is $Nm$.
\end{proof}

\subsection{Multiplicity one results}\label{multone}
Let $g$ and $\fkp$ be as in \ref{newform-gal}, and let $N=N^+N^-$ be a factorization into coprime integers such that $p\mid\mid N^+$ and
$N^-$ is a square-free product of an even number of primes. 

Let $\fkm_0\subset \BT=\BT_0(N^+,N^-)$ be the kernel of the reduction modulo $\fkp$ of the map $\BT\rightarrow\sO_{\fkp}$ associated
with $g$. This is a maximal ideal such that $\BT/\fkm_0 \isoarrow k_0$.
It will be important to have a `multiplicity one' result for the $\fkm_0$-adic Tate modules of the Jacobian $J(X_{N^+,N^-})$.

\begin{lem} \label{mult one}
Suppose that
\begin{itemize}
\item[\rm (a)] $\ov\rho$ is irreducible;
\item[\rm (b)] $\ov\rho$ is not finite at $p$;
\item[\rm (c)] $\ov\rho$ is ramified at all $\ell\mid N^-$ such that $\ell\equiv \pm 1 \pmod{p}$.
\end{itemize}
Then there are isomorphisms of $k_0$-representations of $G_\BQ:$
$$
V_{0}\cong J(X_{N^+,N^-})[\fkm_0] \isoarrow A_{0}[\fkp_0].
$$
\end{lem}

\begin{proof} By (a) and \cite[Thms.~1 and 2]{BosLenRib}, 
the semisimplifications of $J(X_{N^+,N^-})[\fkm_0]$ and $A_{0}[\fkp_0]$ are each a direct sum of a finite number of copies of 
$V_{0}$. Since $J(X_{N^+,N^-})[\fkm_0]$ projects onto $A_{0}[\fkp_0]$, to prove the lemma it therefore suffices to prove that 
$J(X_{N^+,N^-})[\fkm_0]$ is two-dimensional over $k_0$. If $N^-=1$, then this is a result of Mazur and Ribet \cite[Thm.~1]{MR}.  For general $N^-$ it follows from the $N^-=1$ case 
together with \cite[Cor.~8.11]{He}\footnote{
The proof in \cite{He} depends on a level-raising result \cite[Lem.~7.1]{He} for which an appeal is made to \cite{DiamondTaylor}. However,
as noted before the proof of Lemma \ref{raise level lem}, this reference does not cover all cases. In particular, the case where the mod $p$ Galois representation associated with $\fkm$ is not finite at $p$ is not included in the results in \cite{DiamondTaylor} or \cite{Diamond-Taylor2}. However,
Lemma \ref{raise level lem} provides the necessary result for this case.}. 
In particular, we need only observe that since $p\geq 5$, (a) implies that $\fkm_0$ is 
not contained in the set denoted $S$ in {\it loc.~cit.}~(as this set consists exactly of those maximal ideals $\fkm$ such that either the corresponding
residual Galois representation $\ov\rho_\fkm$ is reducible or for which $p=2$ or $3$), 
and (c) implies that $\fkm_0$ is controllable (in the terminology
of {\it loc.~cit.}): from (c) it follows that if $\ov\rho_{\fkm_0} \cong \ov\rho$ 
is unramified - or finite - at some $q\mid N^-$ then $q\not\equiv \pm 1\pmod{p}$, so 
$\ov\rho(\Frob_q)$, which has eigenvalues of ratio $q^{\pm 1}$, is not a scalar.
\end{proof}

Let $q\in \Lambda'$ be an admissible prime,
and let $g'$ and $\fkp'$ be as in Lemma \ref{raise level lem} with $m=q$. Let $\fkm_0'\subset \BT_0(N^+,N^-q)$
be the kernel of the reduction modulo $\fkp'$ of the homomorphism $\BT_0(N^+,N^-q)\rightarrow \sO_{g',\fkp'}$
giving the Hecke action on $g'$. Let $\BT' = \BT_0(N^+,N^-q)_{\fkm_0'}$.  It will also be important to have 
a multiplicity one result in the definite case.

\begin{lem}\label{mult one def}
Suppose that
\begin{itemize}
\item[\rm (a)] $\ov\rho$ is irreducible;
\item[\rm (b)] $\ov\rho$ is not finite at $p$;
\item[\rm (c)] $\ov\rho$ is ramified at all $\ell\mid N^-$ such that $\ell\equiv \pm 1 \pmod{p}$.
\end{itemize}
Then $(\CS_{N^+,N^-}\otimes\BZp)_{\fkm_0'}$ is a free $\BT'$-module of rank one. 
\end{lem}

\begin{proof} 
Denote also by $\fkm_0'$ the preimage of $\fkm_0'$ under the projection $\BT_0(N^+q,N^-)\twoheadrightarrow \BT_0(N^+,N^-q)$.
It follows from hypotheses (a)-(c) and Lemma \ref{mult one} that $J(X_{N^+q,N^-})[\fkm_0']\cong V_0$. Under hypotheses (a)-(c),
the freeness asserted in the lemma then follows from an application of Mazur's Principle: this is just \cite[Lem.~6.5]{He}
since the character group $\CX_{N^+q,N^-}$ of the toric part of the mod $q$ special fibre of the Neron model over $\BZ_q$ of $X_{N^+q,N^-}$ 
can be identified with $H_1(\CG,\BZ) = \CS_{N^+,N^-q}$, even as Hecke modules (see \ref{red at p} with $p$ replaced by $q$).
\end{proof}

\section{Kolyvagin classes for a newform $g$}\label{Koly class g}

Let $g$ and $\fkp$ be as in \ref{newform-gal}. Suppose
\begin{itemize}
\item $N=N^+N^-$ is a permissible factorization;
\item $p\mid\mid N^+$;
\item $N^-$ is the product of an even number of primes ($N^-=1$ is allowed);
\item $\ov\rho$ is irreducible;
\item $\ov\rho$ is not finite at $p$;
\item $\ov\rho$ is ramified at all $\ell\mid\mid N^-$ such that $\ell\equiv\pm 1\pmod{p}$.
\end{itemize}

\subsection{The auxiliary newforms $g_m$}\label{new at m}
For each $m\in\Lambda'$ we fix a newform $g_m=g'\in S_2^{\mathrm{new}}(\Gamma_0(Nm))$ and a prime $\fkp_m=\fkp'\subset\sO_{g_m} =
\sO_{g'}$ as in Lemma \ref{raise level lem}. In particular, $\fkp_{m,0}  = \fkp_m\cap \sO_{g_m,0}$ satisfies 
$\sO_{g_m,0}/\fkp_{m,0} = k_0$ and $\ov\rho_{g_m} \cong \ov\rho$ as $k_0$-representations of $G_\BQ$. 
Clearly, the set of admissible primes with respect to $g_m$ and $\fkp_m$ is
just $\Lambda'\backslash\{q|m\}$ and the set of Kolyvagin primes is still $\Lambda$.

We denote by $A_m$ the fixed abelian variety $A_{g_m}$ and by $\ov\rho_m$, $V_m$, and $V_{m,0}$ the 
respective Galois representations $\ov\rho_{g_m}$, $V_{g_m} = A_{g_m}[\fkp_m]$, and $V_{g_m,0}=V_0$.
The factorization $Nm = N^+\cdot N^-m$ is permissible, and if $m\in \Lambda^{',+}$ then we write
$A_{m,0}$ for the optimal quotient $A_{g_m,0}$ of $J(X_{N^+,N^-m})$ associated with $g_m$ as in \ref{optimal qts}, and we 
assume that $A_m$ is $(\fkp_m,\sO_{g_m})$-optimal. It follows from
Lemma \ref{mult one} applied to $g_m$ that 
$$
A_{m,0}[\fkp_{m,0}]\cong V_{m,0} =  V_{0}.
$$

\subsection{The classes $c(n,m)$}
Following \cite[\S3.2]{Z13} we define cohomology classes $c(n,m)\in H^1(K,V_{0})$, indexed by  $n\in\Lambda$ and $m\in\Lambda^{',+}$.

In particular, for $m\in\Lambda^{',+}$, 
$c(n,m)$ is just the class in $H^1(K,A_{m,0}[\fkp_{m,0}]) = H^1(K,V_{0})$ 
derived from the Heegner points $x_{N^+,N^-m}(n)\in X_{N^+,N^-m}(K[n])$ and $y_{A_{m,0}}(n)\in A_{m,0}(K[n])$
just as $c_1(n)$ in \ref{Koly Conj}. In fact, it follows from this construction that $c_1(n)$ is, up to 
$k^\times$-multiple, just the image of $c(n,1)$. 

Following \cite{Z13}, for each $m\in\Lambda^{',+}$ we set
$$
\kappa_m = \{ c(n,m)\in H^1(K,V_{0}) \ : \ n\in \Lambda\}
$$
and call this a mod $p$ Kolyvagin system for $g$.

\section{Cohomological congruences of Heegner points}\label{Heegner-cong}
This section records a key result that makes possible the induction arguments employed to prove the main results in \cite{Z13}. 
To extend those arguments to cases where $p\mid\mid N$, we have to check that certain crucial cohomological congruences can be extended to 
these cases. This requires the versions of Ihara's Lemma in \ref{Ihara} and and the multiplicity one result from \ref{multone}.

Let $g$ and $\fkp$ be as in \ref{Koly class g} along with all the hypotheses and notation introduced therein.

\subsection{Local cohomology away from $p$}

Let $q\nmid N$ be a prime that is inert in $K$. The {\it finite} - or unramified - part of $H^1(K_q,V_0)$ is the $k_0$-subspace
$$
H^1_{fin}(K_q,V_0) = H^1_{ur}(K_q,V_0)  = H^1(\Gal(K_{q}^{ur}/K_q),V_0) \subset H^1(K_q,V_0),
$$
where $K_{q}^{ur}$ is the maximal unramified extension of $K_q$ and the final inclusion is via the inflation map. 
The {\it singular} part is the quotient
$$
H^1_{sing}(K_q,V_0) = H^1(I_q,V_0)^{\Gal(K_{q}^{ur}/K_q)}.
$$

If $q$ is an admissible prime, then $V_0$ splits uniquely as 
$$
V_0 = k_0\oplus k_0(1)
$$ 
as a $G_{K_q}$-module. 
In the resulting decomposition 
$$
H^1(K_q,V_0) = H^1(K_q,k_0)\oplus H^1(K_q,k_0(1)),
$$ 
$H^1(K_q,k_0)$ and $H^1(K_q,k_0(1))$ are both one-dimensional $k_0$-spaces, 
$H^1_{fin}(K_q,V_0) = H^1(K_q,k_0)$,
and $H^1(K_q,k_0(1))$ projects isomorphically onto $H^1_{sing}(K_q,V_0)$.
Furthermore, for $m\in \Lambda^{',+}$,
\begin{equation*}
\loc_q c(n,m) \in \begin{cases}
H^1(K_q,k_0) & q\nmid m \\
H^1(K_q,k_0(1)) & q\mid m.
\end{cases}
\end{equation*}
See \cite[\S4.1]{Z13} for references.

\subsection{Local cohomology at $p$} \label{loc at p}
In this section we explain the local properties of the Kolyvagin classes $c(n,m)$ at the primes above $p$.

Let $w$ be a prime of $K$ above $p$. Recall that $p$ splits in $K$, so $\BQp\isoarrow K_w$.
Let $\sL_{w}\subset H^1(K_w,V_0)$ be the image of $A_{0}(K_w)/\fkp_0 A_{0}(K_w)$ under the local Kummer map. 
Recalling that the restriction of $V_0$ to $G_{K_w}$ is an extension 
$$
0 \rightarrow k_0(\chi\ov\alpha^{-1})\rightarrow V_0 \rightarrow k_0(\ov\alpha)\rightarrow 0,
$$
which gives rise to an exact sequence of cohomology groups
$$
H^1(K_w,k_0(\chi\ov\alpha^{-1}))\rightarrow H^1(K_w,V_0)\rightarrow H^1(K_w,k_0(\ov\alpha)),
$$
$\sL_{w}$ can be described as follows:

\begin{lem}\label{lem Lp}
If $\ov\rho$ is not finite at $p$, then 
$\sL_{w}=\ker\{H^1(K_w,V_0) \rightarrow H^1(K_w,k_0(\ov\alpha))\}$.
 Equivalently, $\sL_w=\mathrm{im}\{H^1(K_w,k_0(\chi\ov\alpha^{-1}))\rightarrow H^1(K_w,V_0)\}$.
\end{lem}

\begin{proof}
The proof is almost the same as that of \cite[Lem.~8]{GP}. 
We use the non-archimedean uniformization introduced in the proof of Lemma \ref{ord res lem} to first prove the analogous
claim for $A$ and the representation $V$. 
The claim for $A_{0}$ and $V_{0}$ then follows from the $(\sO,\fkp)$-minimality of $A$ and the irreducibility of $V$.

Recall that there is a $G_\BQp$-parameterization
$$
0\to X \to T(\ov\BQ_p)\to A(\ov\BQ_p)\to 0,
$$
where $T=\Hom(Y,\mathbb{G}_m)$ is a torus that splits over an at-most-quadratic unramified extension, 
$X$ and $Y$ are free $\BZ$-modules on which $G_\BQp$ acts via $\alpha$ and which are also locally-free $\sO$-modules of rank one. 
In particular, $H^1(K_w,T)$ is a 2-group, and 
$T[\fkp]$ is identified with the (unique) line $k(\chi\ov\alpha^{-1})$ in $A[\fkp]=V$. Let
$$
T'=T/T[\fkp] \ \ \text{and} \ \ A'=A/A[\fkp].
$$
(In \cite{GP} they are denoted by $\fkp^{-1}\otimes_{\CO} T$ and $\fkp^{-1}\otimes_{\CO}A$, respectively.)
We have a commutative diagram:
$$
\xymatrix{ T'(K_w) \ar[r]\ar[d]& H^1(K_w,T[\fkp])\ar[d]  \\
A'(K_w) \ar[r] & H^1(K_w,A[\fkp])  ,}
$$
where the horizontal arrows are the Kummer maps. The cokernel of the top horizontal map injects into $H^1(K_w,T)$, which is a $2$-group.
As $p$ is odd, it follows that the top horizontal map is surjective. The left vertical arrow is also surjective. It follows that the image of $A'(K_w)$ in  $H^1(K_w,A[\fkp])$ is equal to the image of $H^1(K_w,T[\fkp])=H^1(K_w,k(\chi\ov\alpha^{-1}))$. This proves the desired equality.
\end{proof}

The following corollary is immediate: just apply the lemma with $A_{m,0}$ in place of $A_{0}$.

\begin{cor}\label{lem Lp cor}
Let $m\in\Lambda^{',+}$. The image of the Kummer map $A_{m,0}(K_w)/\fkp_{m,0}A_{m,0}(K_w)\subset H^1(K_w,V_0)$
is $\sL_{w}$.
\end{cor}

\begin{lem}\label{lem c(n) p}
Assume that $\ov{\rho}$ is not finite at $p$. For any $m\in\Lambda^{',+}$ and $n\in\Lambda$,
$$
\loc_w c(n,m)\in \sL_{w}.
$$
\end{lem}

\begin{proof}
By Corollary \ref{lem Lp cor} it suffices to show that $\loc_w c(n,m) \in A_{m,0}(K_w)/\fkp_{m,0}A_{m,0}(K_w)$. 
We have a commutative diagram
$$
\begin{tikzcd}
& H^1(K, A_{m,0}[\fkp_{m,0}]) \arrow{r}{d}\arrow{d}{\loc_w} & H^1(K,A_{m,0})\arrow{d}{\loc_w} \\
A_{m,0}(K_w)/\fkp_{m,0}A_{m,0}(K_w)\arrow{r} & H^1(K_w,A_{g_m,0}[\fkp_{m,0}]) \arrow{r} & H^1(K_w,A_{m,0}).
\end{tikzcd}
$$
Let $d(n,m)\in H^1(K, A_{m,0})[\fkp_{m,0}]$ be the image of $c(n,m)$ under $d$.
 It suffices to show that the image $d_w(n,m)=\loc_w d(n,m)$ of $\loc_wc(n,m)$ in $H^1(K_w,A_{m,0})$ is zero. 
 
 By construction, the restriction of $c(n,m)$ to $H^1(K[n],A_{m,0}[\fkp_{m,0}])$ belongs to the image of $A_{m,0}(K[n])$,
 from which it follows that the image of $d(n,m)$ in $H^1(K[n],A_{m,0})$ restricts to zero in $H^1(K[n]_u,A_{m,0})$ for 
 any place $u$ of $K[n]$. In particular, the restriction of $d_w(n,m)$ to $H^1(K[n]_{w'},A_{m,0})$ is zero for all $w'\mid w$. 
 As $K[n]/K$ is unramified at $w$, it follows that $d_w(n,m)$ belongs to $H^1_{ur}(K_w,A_{m,0})$, the subgroup of 
 unramified classes.
By \cite[Prop.~I.3.8]{M86}, $H^1_{ur}(K_w,A_{m,0})$ injects into $H^1(K_w,\Phi_{m,0})$,
where $\Phi_{m,0}$ is the component group of the N\'eron model of $A_{m,0}$ over $\BZp$. 
Since $d_w(n,m)$ is $\fkp_{m,0}$-torsion, we conclude that 
$d_w(n,m)=0$ if $H^1(K_w,\Phi_{m,0})[\fkp_{m,0}]=0$. But this last vanishing follows from $\Phi_{m,0}[\fkp_{m,0}]=0$ 
(which holds since $\ov\rho_m\cong \ov\rho$ is not finite at $p$; see Corollary \ref{finite at p}) and the natural surjection 
$H^1(K_w,\Phi_{m,0}[\fkp_{m,0}])\twoheadrightarrow H^1(K_w,\Phi_{m,0})[\fkp_{m,0}]$.
\end{proof}

\subsection{Cohomological congruences}\label{coh cong}

\begin{thm} \label{coh cong thm}
Suppose Hypothesis $\heart$ holds for $(g,\fkp,K)$
with $N^-$ a product of an even number of primes and $\ov\rho$ is not finite at $p$.
Let $m\in \Lambda^{',+}$, and let $q_1$ and $q_2$ be admissible primes not dividing $m$.
Then
$$
\loc_{q_1} c(n,m)\in H^1(K_{q_1},k_0) \ \ \text{and} \ \ \loc_{q_2} c(n,mq_1q_2)\in H^1(K_{q_2},k_0(1)),
$$
and $\loc_{q_1}c(n,m)$ is non-zero if and only if $\loc_{q_2} c(n,mq_1q_2)$ is non-zero. 
\end{thm}

\begin{proof}
The proof of \cite[Thm.~4.3]{Z13} carries over with just a few modifications.
Those modifications amount to:
\begin{itemize}
\item using the versions of Ihara's lemma proved in \ref{Ihara} when adapting
the arguments from \cite[Thm.~6.2]{BD-ACMC} to deduce the expression
for $\loc_{q_1}c(n,m)$, and 
\item using the multiplicity one results from \ref{multone} in the argument to
compare the expression for $\loc_{q_2} c(n,mq_1q_2)$ with the one for
$\loc_{q_1}c(n,m)$.
\end{itemize}
\end{proof}

\section{Selmer groups and Rank-lowering}

Let $g$ and $\fkp$ be as in \ref{Koly class g} along with all the hypotheses and notation introduced therein.
In particular, $g_m$, $\fkp_m$, and $A_m$ are as in \ref{new at m}.

\subsection{Selmer groups} 
We recall the mod $\fkp_m$ and $\fkp_m$-adic Selmer groups of $A_m$
over a number field $M$.  

For any place $w$ of $M$ let
$$
\sL_{w,A_m} = \im\{A_m(M_w)\stackrel{\delta_w}{\rightarrow} H^1(M_w,A_m[\fkp_m])\}
$$
and
$$
\CL_{w,A_m} = \im\{A_m(M_w)\otimes\BQp/\BZp \stackrel{\delta_w}{\rightarrow} H^1(M_w,A_m[\fkp_m^\infty])\},
$$
where $\delta_w$ is the Kummer map. Note that $\CL_{w,A_m}=0$ if $w\nmid p$. Then 
$$
\Sel_{\fkp_m}(A_m/M) = \{ c\in H^1(M,A_m[\fkp_m]) \ : \ \loc_w c \in \sL_{w,A_m} \ \forall\text{ places $w$ of $M$}\},
$$
and
$$
\Sel_{\fkp_m^\infty}(A_m/M) = \{ c\in H^1(M,A_m[\fkp_m^\infty]) \ : \ \loc_w c \in \CL_{w,A_m} \
 \forall\text{ places $w$ of $M$}\}.
$$
The natural map $\Sel_{\fkp_m}(A_m/M)\rightarrow \Sel_{\fkp_m^\infty}(A_m/M)[\fkp_m]$ is a surjection with kernel the image of
$A_m(M)[\fkp_m]$. In particular, it is an isomorphism for $M=\BQ$ since $A_m(\BQ)[\fkp_m] = A_m[\fkp_m]^{G_\BQ}=0$ by assumption.

The following proposition, which is just \cite[Thm.~5.2]{Z13}, aids in the comparison of these Selmer groups.

\begin{prop}[{\cite[Thm.~5.2]{Z13}}]\label{Sel loc}
Suppose Hypothesis $\heart$ holds for $(g,\fkp,K)$. 
For each place $w$ of $K$, the local condition $\sL_{w,A_m}$ has a $k_0$-rational structure: there
exists a $k_0$-subspace $\sL_{w,m,0}\subset H^1(K_w,V_0)$ such that
$\sL_{w,m,0}\otimes_{k_0} k = \sL_{w,A_m}$ in $H^1(K_w,V_0)\otimes_{k_0}k_m = H^1(K_w,A_m[\fkp_m])$.
Furthermore, for an admissible prime $q\nmid m$,
$$
\sL_{w,m,0} = \sL_{w,{mq},0} \ \ \forall w\neq q, 
$$
$$
\sL_{q,m,0} = H^1(K_q,k_0), \ \  \sL_{q,{mq},0} = H^1(K_q,k_0(1)).
$$
\end{prop}

\begin{proof}
The proof in \cite{Z13} goes through with only one addition: the equality of the $k_0$-structures at the primes
above $p$ uses Corollary \ref{lem Lp cor} (in fact, this equality was essentially used in the proof of Lemma \ref{lem c(n) p}).
\end{proof}

\subsection{Rank-lowering}
The following proposition is just \cite[Prop.~5.4]{Z13}, proved using Proposition \ref{Sel loc} and Tate duality; the proof goes through unchanged.
It is a key to the induction arguments used to prove the main results in \cite{Z13} and hence also of this paper.

\begin{prop}[{\cite[Prop.~5.4]{Z13}}]\label{rank lower}
Let $q\nmid m$ be an admissible prime. If $\loc_{q}:\Sel_{\fkp_m}(A_m/K)\rightarrow H_{fin}^1(K_q,A_m[\fkp_m])\cong k_m$ is surjective (equivalently, non-trivial), then 
$$
\dim_{k_{mq}}\Sel_{\fkp_{mq}}(A_{mq}/K) = \dim_{k_m} Sel_{\fkp_m}(A_m/K) - 1.
$$
\end{prop}

The usefulness of the preceding proposition is manifest in the light of the following lemma.

\begin{lem}[{\cite[Lem.~7.3]{Z13}}]\label{rank lower lem} 
Suppose Hypothesis $\club$ holds for $\ov\rho$.
For each class $c\in H^1(K,V_0)$, there exists a positive density of admissible primes $q$ such that $\loc_q c \neq 0$. 
\end{lem}

For use in the inductive arguments employed to prove the main result, we record the following lemma, which replaces
\cite[Lem.~8.1]{Z13}.

\begin{lem}\label{indep classes} Suppose $\ov\rho$ is irreducible and its image contains a nontrivial homothety.
Let $c_1$, $c_2$ be two $k_0$-linear independent elements in $H^1(K,V_0)$. Then there exists a positive density of primes 
$\ell\in \Lambda$ such that 
$$
\loc_\ell c_i \neq 0,\quad i=1,2.
$$
\end{lem}

\begin{proof}
We may assume that both $c_1,c_2$ are eigenvectors under the action of $\Gal(K/\BQ)$.
Then the lemma follows from the proof of \cite[Lem.~1.6.2]{Ho-KS}.
Indeed, we only need to consider the case in {\it loc.~cit.}~where $k=1$.
Since the image of $\ov\rho$ contains a nontrivial homothety,
$H^1(K,V)\simeq H^1(L,V)^{\Gal(L/K)}$, where $L$ is as in {\it loc.~cit.}; this replaces Hypothesis H.2 in the proof.
Hypotheses H.1 and H.5(a) are similarly satisfied under the hypotheses of the lemma.
Finally, we need only note that the proof of \cite[Lem.~1.6.2]{Ho-KS} does not need to assume that $c_1$ and $c_2$ have
the different eigenvalues under the action of $\Gal(K/\BQ)$.
\end{proof}

\section{Special value formulas}

Let $g$ and $\fkp$ be as in \ref{newform-gal}.

\subsection{Tamagawa factors}
Let $M$ be a number field and $w$ a finite place of $M$.
Given an abelian variety $\CA$ over $M_w$, let $\sA$ be its N\'eron model and let 
$\Phi_\CA= \pi_0(\sA_0)$ be the group of connected components of the special fibre $\sA_0$ of $\sA$; this is 
an \'etale abelian group scheme over $\BF_w$, or, equivalently, a finite abelian group with a $\Gal(\ov \BF_w/\BF_w)$-action. Then
$$
\Phi_\CA(\BF_w) = H^0(\BF_w,\pi_0(\sA_0)).
$$
Suppose the ring $\sO$ acts on $\CA/M_w$. Then it also acts on $\Phi_\CA(\BF_w)$, and we set
$$
t(\CA/M_w) = \lg_{\sO_{\fkp}} \Phi_\CA(\BF_w)_{\fkp}.
$$
Note that 
$$
\#\Phi_\CA(\BF_w)_{\fkp} = (\# k)^{t(\CA/M_w)}.
$$

\begin{lem}\label{tam coh lem}
Suppose $w\nmid p$. Then 
$$
t(\CA/M_w) = \lg_{\sO_{\fkp}} H^1_{ur}(M_w, \CA[\fkp^\infty])=\lg_{\sO_{\fkp}} H^1(\BF_w, \CA[\fkp^\infty]^{I_w}).
$$
\end{lem}

\begin{proof} We have 
\begin{equation*} \begin{split}
\lg_{\sO_{\fkp}} \Phi_\CA(\BF_w)_\fkp & = \lg_{\sO_{\fkp}} \Phi_\CA(\BF_w)[\fkp^\infty]  \\
& = \lg_{\sO_{\fkp}} H^0(\BF_w, \pi_0(\sA_0))[\fkp^\infty] \\
& = \lg_{\sO_{\fkp}} H^1(\BF_w, \pi_0(\sA_0))[\fkp^\infty] \\
& = \lg_{\sO_{\fkp}} H^1(\BF_w, \sA(\ov M_w^{I_w}))[\fkp^\infty] \\
& = \lg_{\sO_{\fkp}} H^1(\BF_w, \CA(\ov M_w)^{I_w})[\fkp^\infty] \\
& = \lg_{\sO_{\fkp}} H^1(\BF_w, \CA[\fkp^\infty]^{I_w}).
\end{split}
\end{equation*}
The first two equalities follow from the definitions, the third equality follows from $\pi_0(\sA_0)$ having finite order, 
the fourth follows from \cite[Prop.~I.3.8]{M86}, the fifth from the basic properties of N\'eron models, 
and the sixth is an easy consequence of $\fkp\mid p$ and the assumption that~$w\nmid~p$.
\end{proof}

We record a simple corollary for the abelian variety $A$ associated with the newform $g$:
\begin{cor}\label{tam coh cor} 
Let $A^K$ be the $K$-twist of $A$. Let $\ell$ be a rational prime.
$$
\sum_{w|\ell} t(A/K_w) = t(A/\BQ_\ell) + t(A^K/\BQ_\ell).
$$
\end{cor}

\begin{proof} Since $A^K \cong A$ over $K$, there is nothing to prove if $\ell$ splits in $K$.
If $\ell$ is inert or ramified in $K$, let $w$ be the unique place of $K$ over $\ell$. Then, since $p\neq 2$, the restriction map 
$$
H^1(\BF_\ell,A[\fkp^\infty]^{I_\ell}) \oplus H^1(\BF_\ell,A^K[\fkp^\infty]^{I_\ell}) \rightarrow H^1(\BF_w,A[\fkp^\infty]^{I_w})
$$
is an isomorphism, from which the desired equality follows as $\ell\neq p$ (since $p$ splits in $K$).
\end{proof}

We also record the following consequence of Lemma \ref{finite at p} for the Tamagawa factors at primes above $p$.

\begin{lem}\label{tam p zero}
If $\ov\rho$ is not finite at $p$, then $t(A/K_w) = 0$ for all $w\mid p$.
\end{lem}

\subsection{The canonical periods of $g$}\label{Periods}

We recall the definition of the periods $\Omega^\pm_{g,\Gamma}$ associated to $g$, $\fkp$, and the congruence subgroup
$\Gamma = \Gamma_0(N)$ or $\Gamma_1(N)$.  

Let $\BT_\Gamma$ be the  Hecke algebra for level $\Gamma$
generated over $\sO_{(\fkp)}$ by the actions of the usual Hecke operators on the space $S_2(\Gamma)$ of weight $2$ cuspforms of level $\Gamma$;
if $\Gamma = \Gamma_0(N)$, then $\BT_\Gamma$ is just $\BT_{N,1}\otimes\sO_{(\fkp)}$. Let $\phi_\Gamma:\BT_\Gamma\twoheadrightarrow \sO_{(\fkp)}$ 
be the $\sO_{(\fkp)}$-linear homomorphism giving the
action of the Hecke operators on the newform $g$. Then we have a factorization
$$
\phi_{\Gamma_1(N)}: \BT_{\Gamma_1(N)}\twoheadrightarrow \BT_{\Gamma_0(N)}
\stackrel{\phi_{\Gamma_0(N)}}{\twoheadrightarrow} \sO_{(\fkp)},
$$
where the first arrow is the canonical surjection (induced by the inclusion $S_2(\Gamma_0(N))\subset S_2(\Gamma_1(N))$.
Let $\fkP_\Gamma = \ker(\phi_\Gamma)$.

Recall the Eichler-Shimura map:
\begin{equation*}
\mathrm{Per}_\Gamma: S_2(\Gamma) \hookrightarrow H^1(\Gamma,\BC), \ \ \ 
f\mapsto (\gamma \mapsto \int_\tau^{\gamma(\tau)} f(z)dz).
\end{equation*}
Let $\omega_{g,\Gamma} = \mathrm{Per}_\Gamma(g)$. We decompose $\omega_{g,\Gamma}$ as $\omega_{g,\Gamma} = 
\omega_{g,\Gamma}^++\omega_{g,\Gamma}^-$,
 according to the decomposition $ H^1(\Gamma,\BC)=H^1(\Gamma,\BC)^+\oplus H^1(\Gamma,\BC)^-$ under
 the action of conjugation by $\left( \smallmatrix 1&\\&-1\endsmallmatrix\right)$, the superscript `$\pm$' denoting the subspace on
 which the action is just $\pm 1$ (this corresponds to the action of complex conjugation on the Betti cohomology of the modular curve).
 The $\fkP_\Gamma$-torsion $H^1(\Gamma,\sO_{(\fkp)})^\pm[\fkP_\Gamma]$ is a free $\sO_{(\fkp)}$-module of rank one
 (the superscript `$\pm$' means the same as before), and we fix
 an $\sO_{(\fkp)}$-generator $\gamma^\pm_{g,\Gamma}$; this is uniquely determined up to an $\sO_{(\fkp)}^\times$-multiple.
 We define the periods $\Omega^\pm_{g,\Gamma}\in\BC^\times$ (up to $\sO_{(\fkp)}^\times$-multiple) by 
\begin{align}\label{eqn omega g}
\omega_{g,\Gamma}^\pm=\Omega_{g,\Gamma}^\pm \gamma_{g,\Gamma}^\pm.
\end{align}

The periods $\Omega_{g,\Gamma_1(N)}^\pm$ are often used in the literature (for example, in \cite{SU} and \cite{Smult}), while 
we will need to use the periods $\Omega_{g,\Gamma_0(N)}^\pm$. The following lemma makes the passage between results 
using these periods easy in many cases.
 
\begin{lem}\label{per lemma}
If $\ov\rho$ is irreducible, then, up to $\sO_{(\fkp)}^\times$-multiple,
$$
\Omega_{g,\Gamma_0(N)}^\pm=\Omega_{g,\Gamma_1(N)}^\pm.
$$
\end{lem}

\begin{proof}
Let $\Sigma_N$ be the Shimura subgroup of $J_1(N)$, i.e., the kernel of the natural map of Jacobians $J_0(N)\to J_1(N)$
 induced (via $\Pic^0$ functoriality) by the natural degeneracy map $X_1(N)\to X_0(N)$. 
 The group $\Sigma_N$ is Eisenstein (\cite{R-Shimura}). Let $\fkm$ be the kernel of the composition 
 $\BT_{\Gamma_1(N)} \to \sO_{(\fkp)}\to \sO_{(\fkp)}/\fkp = k$; the image of $\fkm$ in $\BT_{\Gamma_0(N)}$
 is just the similarly defined maximal ideal.
 As $\Sigma_N$ is Eisenstein, it follows from the irreducibility of $\ov\rho = \ov\rho_\fkm$ 
 (which is just the Galois representation associated with $\fkm$) 
 and the duality between cohomology groups and the Tate-modules of Jacobians
 that after localization at $\fkm$ we have an injection
$$
H^1(\Gamma_0(N), R)_{\fkm}\hookrightarrow H^1(\Gamma_1(N),R)_{\fkm}
$$
for $R=\sO_{(\fkp)}$ and for $R=k$. 
Note that $H^1(\Gamma,R)[\mathfrak{P}_\Gamma]^\pm=H^1(\Gamma,R)_\fkm[\mathfrak{P}_\Gamma]^\pm$.
This implies that the injection above for $R=\sO_{(\fkp)}$ induces an isomorphism of free $R$-modules of rank one:
$$
H^1(\Gamma_0(N),R)[\mathfrak{P}_{\Gamma_0(N)}]^\pm \simeq H^1(\Gamma_1(N),R)[\mathfrak{P}_{\Gamma_1(N)}]^\pm.
$$
In particular, the classes $\gamma_{g,\Gamma}^\pm$ are identified up to $\sO_{(\fkp)}^\times$-multiple.
Furthermore, under the natural inclusion $H^1(\Gamma_0(N), \BC)\incl H^1(\Gamma_1(N),\BC)$  
(which extends\footnote{Implicit in the consideration of $g$ as a holomorphic function on $\fkh$ is
an embedding of $F$ into $\BC$.} that for $R=\sO_{(\fkp)}$)
the cohomology classes $\omega_{g,\Gamma}$ attached to $g$ are identified.  
The lemma then follows immediately from \eqref{eqn omega g}.  
\end{proof}

In light of the preceding lemma, when $\ov\rho$ is irreducible we let
\begin{equation}\label{eqn omega can}
\Omega_g^\pm  = \Omega_{g,\Gamma_0(N)}^\pm.
\end{equation}

\subsection{Congruence numbers and congruence periods}

For any factorization $N=N^+ N^-$ with $N^+$ and $N^-$ coprime, $p\mid\mid N^+$, and $N^-$ squarefree, we 
let $\eta_g(N^+,N^-) \in \sO_{\fkp}$ be a generator of the new-at-$N^-$-congruence ideal for $g$.
In particular, let $\pi:\BT_0(N^+,N^-)\otimes_\BZp\sO_\fkp \twoheadrightarrow \sO_{\fkp}$ be the $\sO_\fkp$-linear map giving the Hecke action on the eigenform 
$g$ and let $\fkm = \pi^{-1}(\fkp)$ be the associated maximal ideal. Let $\BT(N^+,N^-)_\fkm$ be the localization of
$\BT_0(N^+,N^-)\otimes_\BZp\sO_\fkp$ at $\fkm$; this is just the completion of $\BT_{N^+,N^-}\otimes\sO_{(\fkp)}$ at the maximal ideal
also denoted $\fkm$ in the proof of Lemma \ref{per lemma}. Then
$$
(\eta_g(N^+,N^-))=\pi(\mathrm{Ann}_{\BT(N^+,N^-)_\fkm}(\ker\pi)).
$$ 
We set
$$
\eta_g= \eta_g(N,1).
$$
This is just the usual congruence number for $g$.

We now fix an isomorphism $\BC\cong\ov\BQ_p$ so that valuation induced on the subfield $F\subset\BC$ is 
that associated with $\fkp$.
The congruence period (or Hida period) of $g$ is then defined to be 
$$
\Omega_g^{cong}=\frac{\pair{g,g}}{\eta_g} \in \ov\BQ_p^\times.
$$
where 
$$
\pair{g,g'}=4\pi^2 i \int_{X_0(N)(\BC)}\omega_g\wedge \ov{\omega_{g'}} = 8\pi^2 \int_{\Gamma_0(N)\backslash\fkh} g(z)\ov{g'(z)}dxdy
$$ 
is the Petersson inner product.
Here, $\omega_g$ denotes the holomorphic differential on the modular curve $X_0(N)$ that is the unique holomorphic
extension of the differential on the open modular curve $Y_0(N)$ 
that pulls back to $g(z)dz$ under the usual complex uniformization $\Gamma_0(N)\backslash\fkh \isoarrow Y_0(N)$.
This is identified with the $\omega_g$ in \ref{Periods} via the deRham map
$$
H^0(X_0(N)_{/\BC},\Omega^1)\isoarrow H^1(X_0(N),\BC)^{1,0}\hookrightarrow H^1(Y_0(N),\BC) = H^1(\Gamma_0(N),\BC).
$$

\begin{lem}\label{period comp lem}
If $\ov\rho$ is irreducible, then, up to $\sO_\fkp^\times$-multiple,
$$
\Omega_g^{cong}= i (2\pi i)^2\Omega_g^+\Omega_g^-.
$$
\end{lem}

\begin{proof} 
This is essentially proved in \cite[\S4.4]{DDT}. More precisely, via the natural Hecke-equivariant identification $H^1(X_0(N),\sO_\fkp)_\fkm = H^1(\Gamma_0(N),\sO_\fkp)_\fkm$,
we can take the $\{x,y\}$ of \cite[Cor.~4.19]{DDT} to be $\{\gamma_{g,\Gamma_0(N)}^+,\gamma_{g,\Gamma_0(N)}^-\}$. 
Then, by \cite[Cor.~4.19]{DDT} and the first displayed equation in the proof of \cite[Thm.~4.20]{DDT}, 
$$
\eta_g \det A = \pair{\gamma_{g,\Gamma_0(N)}^+,\gamma_{g,\Gamma_0(N)}^-}_H\det A = \int_{X_0(N)} 2\pi i\omega_g\wedge\ov{2\pi i\omega_{wg}} 
 = -i\pair{g,wg} = \pm i\pair{g,g},
 $$
 where $A= 2\pi i \left(\smallmatrix \Omega_g^+ & \Omega_g^+ \\ \Omega_g^- & - \Omega_g^-\endsmallmatrix\right)\in \GL_2(\BC)$ is such that
 $2\pi i (\omega_g,\bar \omega_g) = (\gamma_{g,\Gamma_0(N)}^+,\gamma_{g,\Gamma_0(N)}^-)A$, $w = w_N$ is the Atkin-Lehner involution, and 
 we have used that $g$ has real Fourier coefficients (so $g^c = g$) and $wg = \pm g$. As $\det A = 4\pi^2\Omega_g^+\Omega_g^-$,
 the lemma follows.
 \end{proof}

For comparison of various special value formulas we record the following relation between the canonical periods of
$g$ and its $K$-quadratic twist $g^K$ (that is, the newform associated with the twist of $g$ by the quadratic
Dirichlet character $\chi_K$ associated with the extension $K/\BQ$). Note that $F_{g^K} = F_{g}$. If we also define 
the canonical periods of $g^K$ with the respect to the prime $\fkp$ of $\sO_{g^K} = \sO_g$, then we have the following.

\begin{lem}\label{can period twist} If $\ov\rho$ is irreducible, then, up to $\sO_{\fkp}^\times$-multiple,
$$
\Omega_{g^K}^\pm = \Omega_{g}^\mp.
$$
\end{lem}

\noindent  While this is certainly well-known to experts, we include a proof for lack of a convenient reference. Our proof makes use of Lemma \ref{period comp lem}.
For an ordinary form $g$, which includes the cases considered for the main results of this paper, 
it is possible to avoid this - and the assumption that $\ov\rho$ is irreducible - and use instead properties of the $p$-adic $L$-function of $g$, but we do not go into this here.

\begin{proof} We first show that $\Omega_{g^K}^\pm$ is an $\sO_{\fkp}$-multiple of $\Omega_{g}^\mp$; this part does not use that $\ov\rho$ is irreducible.
For any $\BZ$-algebra $R$ there is a homomorphism
\begin{equation*}\begin{split}
H^1(\Gamma_0(N),R) & \rightarrow H^1(\Gamma_0(ND^2),R), \\
\varphi\mapsto (\gamma\mapsto \sum_{a\in(\BZ/D\BZ)^\times} \chi_K & (a)\varphi(\left(\smallmatrix 1 & -a/D \\ 0 & 1 \endsmallmatrix\right)\gamma\left(\smallmatrix 1 & a/D \\ 0 & 1\endsmallmatrix\right))).
\end{split}
\end{equation*}
Under this homomorphism, for $R=\BC$, $\omega_g$ (resp.~$\omega_g^\mp$) gets mapped to $\tau(\chi_K)\omega_{g^K}$ (resp.~$\tau(\chi_K)\omega_{g^K}^\pm$; since $\chi_K$
is odd the $\mp$-submodule gets mapped into the $\pm$-submodule), where
$\tau(\chi_K)$ is the Gauss sum attached to $\chi_K$. For $R=\sO_{(\fkp)}$, $\gamma_{g,\Gamma_0(N)}^\mp$ gets mapped to an $\sO_{(\fkp)}$-multiple of $\gamma_{g^K,\Gamma_0(ND^2)}^\pm$. 
It follows that
$\Omega_{g^K}^\pm$ is an $\sO_{(\fkp)}$-multiple of $\Omega_g^\mp/\tau(\chi_K)$. However, since $p$ splits in $K$ and $\tau(\chi_K)^2 = -D$, $\tau(\chi_K)\in \BZp^\times$.
It follows that $\Omega_{g^K}^\pm$ is an $\sO_\fkp$-multiple of $\Omega_g^\mp$: $\Omega_{g^K}^\pm = a_\mp \Omega_g^\mp$.

To complete the proof of the lemma, it suffices to show that $a_+a_- \in \sO_\fkp^\times$. To show this last inclusion, we exploit Lemma \ref{period comp lem}.
Since $(D,pN)=1$, twisting by $\chi_K$ shows that the congruence ideal $(\eta_{g^K})$ is just the congruence ideal 
measuring mod $p$ congruences between $g$ and forms of level $\Gamma_0(N\prod_{\ell\mid D}\ell^2)$, in the sense that the $n$th Fourier coefficients for $(n,D)=1$ are congruent.
It then follows easily from the calculations used to prove \cite[Thm.~4.20]{DDT} that if $\ov\rho$ is irreducible, then, up to $\sO_\fkp^\times$-multiples, 
$$
\frac{\eta_{g^K}}{\eta_g} = \prod_{\ell\mid D}(1-\alpha(\ell)^2\ell^{-2})(1-\beta(\ell)\ell^{-2})(1-\ell^{-1}) = \frac{\pair{g^K,g^K}}{\pair{g,g}}.
$$
Here $\alpha(\ell)$ and $\beta(\ell)$ are the roots of the Hecke polynomial $x^2-a(\ell)x+\ell$.
It follows in particular that, up to $\sO_\fkp^\times$-multiple, $\Omega_g^{cong} = \Omega_{g^K}^{cong}$. That $a_+a_-$ belongs to $\sO_\fkp^\times$ then follows from this together with Lemma \ref{period comp lem} applied to both $g$ and $g^K$.
\end{proof}

\subsection{The BSD formula in rank zero}\label{BSD0}
We recall the $\fkp$-part of the BSD formula for $L(g,1)$:

\begin{thm}[{\cite[Thm.~B]{Smult}}]\label{thm BSD 0} Suppose
\begin{itemize}
\item[\rm (a)] $\ov\rho$ is irreducible;
\item[\rm (b)] there exists a prime $q\neq p$ such that $q\mid\mid N$ and $\ov\rho$ is ramified at $q$;
\item[\rm (c)] if $p\mid\mid N$ and $a(p)=1$, then the Mazur-Tate-Teitelbaum $\fkL$-invariant $\fkL(\CV)$ of $\rho$ is non-zero.
\end{itemize}
Then
$$
\ord_\fkp\left(\frac{L(g,1)}{2\pi i \Omega_g^+}\right) = \lth_{\sO_{\fkp}} \Sel_{\fkp^\infty}(A/\BQ) + \sum_{\ell|N} t(A/\BQ_\ell).
$$
\end{thm}

\noindent{See \ref{L-inv} for the definition of $\fkL(\CV)\in F_\fkp$.

Since $p$ splits in $K$, the $K$-twist $g^K$ of $g$ also satisfies the hypotheses of Theorem \ref{thm BSD 0}. The $K$-twist of $g$ is its twist
by the primitive quadratic Dirichlet $\chi_K$ of conductor $D$. The abelian variety associated with $g^K$ is the $K$-twist $A^K$ of $A$. 
From the decomposition of $\Sel_{\fkp^\infty}(A/K)$ under the action of 
$\Gal(K/\BQ)$ into the sum of the respective Selmer groups for $A$ and its $K$-twist $A^K$ and from Corollary \ref{tam coh cor}, we deduce from Theorem \ref{thm BSD 0}:

\begin{thm}\label{thm BSD K0}
 Suppose
\begin{itemize}
\item[\rm (a)] $\ov\rho$ is irreducible;
\item[\rm (b)] there exists a prime $q\neq p$ such that $q\mid\mid N$ and $\ov\rho$ is ramified at $q$;
\item[\rm (c)] if $p\mid\mid N$ and $a(p)=1$, then the $\fkL$-invariant $\fkL(\sV)$ of $\rho$ is non-zero.
\end{itemize}
Then
$$
\ord_\fkp\left(\frac{L(g/K,1)}{i\Omega_g^{cong}}\right) = \lth_{\sO_{\fkp}} \Sel_{\fkp^\infty}(A/K) + \sum_{w|N} t(A/K_w).
$$
\end{thm}
\noindent Here, the sum is over the places $w$ of $K$ dividing $N$ and  $L(g/K,s) = L(g,s)L(g^K,s)$. 
The key observation reducing this theorem to the preceding is that by Lemma \ref{can period twist}, up $\sO_\fkp^\times$-multiple, $\Omega_{g^K}^+=\Omega_g^-$ and so, by Lemma \ref{period comp lem}, 
up to $\sO_\fkp^\times$-multiple, $i\Omega^{cong} = 2\pi i\Omega_g^+\cdot 2\pi i\Omega_{g^K}^+$.

\subsection{Gross's special value formula and the Gross period}

Suppose $N = N^+N^-$ is an admissible factorization with $N^-$ a product of an odd number of primes. 
Let $\phi_g\in \CS_{N^+,N^-}\otimes\sO_{(\fkp)}$ be an eigenform corresponding to $g$ under the Jacquet-Langlands correspondence, normalized to be non-zero
modulo $\fkp$.  That is, $\phi_g$ is an eigenvector with the same eigenvalues as $g$ and non-zero modulo $\fkp$;
$\phi_g$ is unique up to $\sO_{(\fkp)}^\times$-multiple. 
We view $\phi_g$ as an $\sO_{(\fkp)}$-valued function on $X_{N^+,N^-}$ such that
$\sum_{x\in X_{N^+,N^-}} \phi_g(x) = 0$ (if $\phi_g= \sum_x a_x\cdot x$, then $\phi(x) = a_x$). Then $\phi_g$ clearly extends to a function on $\BZ[X_{N^+,N^-}]$.
Gross's special value formula is then just

\begin{thm}[{\cite[(7--8)]{Vatsal-GZ}}]
Let $x_K= x_{N^+,N^-,K}\in \BZ[X_{N^+,N^-}]$. Then 
$$
\frac{|\phi_g(x_K)|^2}{\pair{\phi_g,\phi_g}} = D^{1/2}(w_K/2)^2\frac{L(g/K,1)}{\pair{g,g}}.
$$
\end{thm}
\noindent Here $w_K$ is the order of $\sO_K^\times$ and $\pair{\phi_g,\phi_g}$ is just the Petersson norm of $\phi_g$ with respect to the counting measure. 
Since $\sO$ is totally real, the expression in the theorem can be rewritten as
\begin{equation}\label{Gross form 1}
\phi_g(x_K)^2 = \frac{L(g/K,1)}{\Omega_g^{Gr}},
\end{equation}
where $\Omega_g^{Gr}$ is the Gross period:
$$
\Omega_g^{Gr} = D^{-1/2}(w_K/2)^{-2}\frac{\pair{g,g}}{\pair{\phi_g,\phi_g}}.
$$
This last quantity is clearly well-defined up to $\sO_{(\fkp)}^\times$-multiple.

To compare \eqref{Gross form 1} with the expression in Theorem \ref{thm BSD K0}, we need to compare the periods $\Omega_g^{cong}$ and $\Omega_g^{Gr}$.
Let
$$
\eta_{g,N^+,N^-} = \frac{\eta_g}{\pair{\phi_g,\phi_g}} = D^{1/2}(w_K/2)^2\frac{\Omega_g^{Gr}}{\Omega_g^{cong}}.
$$
The conclusion of \cite[Theorem 6.4]{Z13} remains true for the forms we consider:

\begin{thm}\label{thm RT} Assume that Hypothesis $\heart$ holds for $(g,\fkp,K)$ with $N^-$ a product of an odd number of prime factors,
and $p\mid\mid N^+$ and that $\ov\rho$ is not finite at $p$.
Then
$$
\ord_\fkp (\eta_{g,N^+,N^-})=\sum_{\ell\mid N^-}\lth_{\sO_\fkp } t(A/K_\ell).
$$
\end{thm}

\begin{proof} 
The proof of \cite[Theorem 6.4]{Z13} carries over. 
The proof goes along the same line as that of \cite[Theorem 6.8]{PW}, and it suffices to verify the analogous statements for \cite[Thm~6.2-6.8]{PW}. 
For this, we note the following:
\begin{itemize}
\item The multiplicity one result needed to carry over the proof of \cite[Theorem 6.2]{PW} is provided by Lemma \ref{mult one}. 
\item \cite[Prop.~6.5]{PW} is from \cite{Kohel}, which does not assume $p\nmid N$ (nor the square-freeness of $N$).
\item In the proof \cite[Prop.~6.7]{PW}, the square-freeness is only used to find a prime factor $r|N$ such that $\ov{\rho}$ is ramified at $r$; Hypothesis $\heart$ ensures such an $r$ exists.
\item  To carry over the proof of \cite[Thm.~6.8]{PW} we need the analog of the last displayed equation from {\it loc.~cit.}, which is 
due to Ribet--Takahashi \cite{RT}, and Takahashi \cite{T}:
$$
\frac{\delta_f(N,1)}{\delta_f(N_1,N_2)}=\sum_{\ell|N_2}\lth_{\sO_\fkp }\Phi(A/K_\ell)_{\fkp },
$$
 for any coprime factorization $N=N_1N_2$ with $N_2$ a square-free product of an even number of primes\footnote{If $N$ is square-free, this is proved by Ribet and Takahash \cite{RT} 
under the assumption that $N_1$ is not a prime. This assumption was removed by Takahashi \cite{T}.}. 
This result does not assume $p\nmid N$. Moreover, if $N$ is not square-free, from the proof of the second assertion of \cite[Thm.~1]{RT}, one may deduce the same formula under either of the following 
two assumptions:
 \begin{itemize}
 \item  there exists a prime $\ell\mid\mid N_1$ such that $\ov\rho$ is ramified at $\ell$ and a different prime $\ell'\mid\mid N_1$;
  \item there exists a prime $\ell\mid\mid N_2$ such that $\ov\rho$ is ramified at $\ell$. 
 \end{itemize}
That one or the other holds for $\ov\rho$ is guaranteed by Hypothesis $\heart$ for $(g,\fkp,K)$.
Note that  Ribet--Takahashi's result (proved for elliptic curves in \cite{RT}) has been extended to 
$\GL_2$-type abelian varieties by Khare \cite{KhareGL2}, and these results do not require $p\nmid N$.
\end{itemize}
The proofs of \cite[Prop.~6.3, 6.4, 6.6]{PW} carry over directly.
\end{proof}

We deduce the following important consequence:

\begin{thm}\label{period Sel thm} 
Assume that Hypothesis $\heart$ holds for $(g,\fkp,K)$ with $N^-$ a product of an odd number of prime factors and $p\mid\mid N$
and that $\ov\rho$ is not finite at $p$.
Suppose also that if $a(p)=1$, then the $\fkL$-invariant $\fkL(\sV)$ of $\rho$ is non-zero. 
Let $x_K = x_{N^+,N^-,K}\in \BZ[X_{N^+,N^-}]$. Then
$$
2\cdot\ord_\fkp(\phi_g(x_K)) = \ord_\fkp(\frac{L(g/K,1)}{\Omega_g^{Gr}}) = \lth_\fkp \Sel_{\fkp^\infty}(A/K).
$$ 
\end{thm}

\begin{proof} By the definitions of $\Omega_g^{Gr}$ and $\Omega_g^{cong}$, we have 
$$
\frac{L(g/K,1)}{\Omega_g^{Gr}} = \sqrt{-D}(w_K/2)^2\frac{L(g/K,1)}{i\Omega^{cong}}\cdot\frac{1}{\eta_{g,N^+,N^-}}.
$$
Since $\sqrt{-D}\in \BZ_p^\times$ (as $p$ splits in $K$) and $w_K/2\in\BZ_p^\times$ (as $p\geq 5$), we then have
by \eqref{Gross form 1} and Theorem \ref{thm RT} that
$$
2\cdot\ord_\fkp(\phi_g(x_K)) = \ord_\fkp(\frac{L(g/K,1)}{i\Omega_g^{cong}}) - \sum_{\ell\mid N^-} t(A/K_\ell). 
$$
Combining this with Theorem \ref{thm BSD K0}, we conclude that
$$
2\cdot\ord_\fkp(\phi_g(x_K)) =  \lth_\fkp\Sel_{\fkp^\infty}(A/K) + 2\sum_{\ell\mid N^+} t(A/\BQ_\ell).
$$
The desired equality is now seen to hold upon noting that part (2) of hypothesis $\heart$ 
ensures that for all $\ell\mid N^+$, $\ell\neq p$,
$H^1(\BF_\ell, A[\fkp^\infty]^{I_\ell}) = 0$ and hence, by Lemma \ref{tam coh lem}, that $t(A/\BQ_\ell)=0$,
and, furthemore, that $t(A/\BQ_p)=0$ by Lemma \ref{tam coh cor} (as $\ov\rho$ is not finite at $p$).
\end{proof}

\subsection{Jochnowitz congruences}
Suppose that $N=N^+N^-$ is a permissible factorization with $N^-$ a product of an even number of primes. 
The {\em Jochnowitz congruence} shows that that the Heegner point $y_K = y_{A_{g},K} \in A_{g}(K)$ is non-torsion 
if a certain $L$-value for a congruent form $g'$ is non-zero modulo $p$. As stated here, this is just \cite[Thm.~6.5]{Z13}
and the proof carries over directly.

\begin{prop}[{\cite[Thm.~6.5]{Z13}}]\label{Jochno cong}
Assume that Hypothesis $\heart$ holds for $(g,\fkp,K)$ with $N^-$ a product of an even number of prime factors and $p\mid\mid N^+$
and that $\ov\rho$ is not finite at $p$.
Let $q\in \Lambda'$ be an admissible prime and $g'$ a newform of level $Nq$ congruent to $g$ as in 
Lemma \ref{raise level lem} with associated prime $\fkp'$. 
Then $\loc_q c(1)\in H^1(K_q,V_0)$ is non-zero if and only if $\frac{L(g'/K,1)}{\Omega_{g'}^{Gr}}$ is non-zero modulo~$\fkp'$.
\end{prop}

\begin{remark}\hfill
\begin{itemize}
\item[(a)] The equivalence with the normalized $L$-value not vanishing arises from \eqref{Gross form 1}: 
as a consequence of the multiplicity one result of Lemma \ref{mult one def}, $\loc_q c(1)$ is shown to be non-zero 
if and only if $\phi_{g'}(x_K)$ is non-zero modulo $\fkp'$.
\item[(b)] As $c(1)$ is the image of $y_{A_0,K} \in A_0(K)$ under the Kummer map 
$$
A_0(K)/\fkp_0A_0(K)\hookrightarrow H^1(K,V_0),
$$ 
this shows that $y_{A_0,K}$, and hence $y_K$, is non-zero and even non-torsion (as $A_0(K)[\fkp_0]=~0$). 
\end{itemize}
\end{remark}

\section{The rank one case}

We explain that the base case (the rank one case) of the induction argument of \cite{Z13} continues to hold for
certain cases where $p\mid\mid N$. Let $g$ and $\fkp$ be as in \ref{newform-gal}.

\subsection{The rank one case} \label{rank one}
This is just the extension of \cite[Thm,~7.2]{Z13} to the cases with $p\mid\mid N$ considered here.

\begin{thm}\label{rank one thm} Suppose $p\geq 5$ and
\begin{itemize}
\item[\rm (a)] Hypothesis $\heart$ holds for $(g,\fkp,K)$ with $N^-$ a product of an even number of primes and $p\mid\mid N^+$;
\item[\rm (b)] Hypotheses $\club$ holds for $\ov\rho$;
\item[\rm (c)] $\ov\rho$ is not finite at $p$;
\item[\rm (d)] if $a(p)=1$ then Hypothesis $\fkL$ holds for $\ov\rho$;
\end{itemize}
If $\dim_k\Sel_\fkp(A/K) = 1$, then 
$c(1)\neq 0$. In particular, the Heegner point $y_K=y_{A,K} \in A(K)$ is non-torsion.
\end{thm}

\begin{proof} By Lemma \ref{rank lower lem} there exists an admissible prime $q$ such that 
$$
\loc_q:\Sel_\fkp(A/K)\twoheadrightarrow H^1_{fin}(K_q,V_0)\cong k_0.
$$
Let $g'$ be a newform of level $Nq$ congruent to $g$ as in Lemma \ref{raise level lem} with associated prime $\fkp'$, and let $k'=\sO_{g'}/\fkp'$.
Let $A'=A_{g'}$ be the abelian variety associated with $g'$. Then by Proposition \ref{rank lower}, 
$$
\dim_{k'} \Sel_{\fkp'}(A'/K) = \dim_k\Sel_\fkp(A/K) - 1 = 0.
$$ 
As $A'[\fkp']$ is irreducible, $\Sel_{\fkp'}(A'/K)$ is the $\fkp'$-torsion of $\Sel_{\fkp^{'\infty}}(A'/K)$ and so 
the latter is also zero. Hypothesis $\heart$ clearly also holds for $(g',\fkp',K)$.
Furthermore, since Hypothesis $\fkL$ holds for $(g,\fkp)$ by assumption if $a_g(p)=1$ and sine
$A'[\fkp'] \cong V_0\otimes_{k_0}k'$, Hypothesis $\fkL$ then also holds for $(g',\fkp')$ if $a_{g'}(p)=1$ 
(as $a_g(p)=1$ if and only if $a_{g'}(p)=1$). 
In particular, if $a_{g'}(p)=1$ then the $\fkL$-invariant for $g'$ is non-zero by Lemma \ref{Linv lem}. 
It then follows from Theorem \ref{period Sel thm} applied to $g'$ that
$\ord_{\fkp'}(\frac{L(g'/K,1)}{i\Omega_{g'}^{cong}})=0$ and
$\phi_{g'}(x_K)$ is non-zero modulo $\fkp'$.
The conclusion of the theorem is then a consequence of the Jochnowitz congruence of Proposition \ref{Jochno cong}.
\end{proof}

\section{The main results}\label{main results}

Theorems \ref{coh cong thm} and \ref{rank one thm} are the keys  to the inductive arguments in \cite{Z13}.
Having shown that they continue to hold, the proofs of the main results of \cite{Z13} carry over, with 
\cite[Lem.~8.1]{Z13} replaced with Lemma \ref{indep classes}.

\subsection{Kolyagin's conjecture for non-finite multiplicative reduction}
We obtain the non-vanishing of the mod $p$ Kolyvagin system $\kappa_1$ in some cases of multiplicative reduction.
This is the analog of \cite[Thms.~9.1, 9.3]{Z13}.

\begin{thm}\label{Kconj thm} Let $p\geq 5$ be a prime. 
Let $g$ and $\fkp$ be as in \ref{newform-gal} and let $K$ be an imaginary quadratic field of discriminant $-D$. 
Suppose
\begin{itemize}
\item[\rm (a)] $(D,N)=1$;
\item[\rm (b)] $p$ splits in $K$ and $p\mid\mid N$;
\item[\rm (c)] Hypothesis $\heart$ holds for $(g,\fkp,K)$ with $N^-$ a product of an even number of primes;
\item[\rm (d)] Hypothesis $\club$ holds for $\ov\rho$;
\item[\rm (e)] $\ov\rho$ is not finite at $p$;
\item[\rm (f)] if $a(p) =1$, then Hypothesis $\fkL$ holds for $\ov\rho$.
\end{itemize}
Then the mod $p$ Kolyvagin system $\kappa_1 = \{c(n,1)\in H^1(K,V_0) \ : \ n\in \Lambda\}$ is non-zero.
In particular, the Kolyvagin system $\kappa^\infty$ is non-zero and, furthermore, $\sM_\infty(g)=0$.
\end{thm}

\subsection{The parity conjecture for non-finite multiplicative reduction}
We also obtain a version of the parity theorem in cases of multiplicative reduction.
This is the analog of \cite[Thm.~9.2]{Z13}.

\begin{thm}\label{parity thm} Under the hypotheses of Theorem \ref{Kconj thm}, both
$\dim_k \Sel_\fkp(A_g/K)$ and the $\sO_{g,\fkp}$-corank of $\Sel_{\fkp^\infty}(A_g/K)$
are odd.
\end{thm}

\section{Theorems for elliptic curves with non-finite multiplicative reduction}
We now explain how Theorems \ref{thm E main}, \ref{thm E BSD}, and \ref{thm E/Q} follow.

\subsection{Proof of Theorem \ref{thm E main}}  Let $E$, $N$, $p$, and $K$ be as in Theorem \ref{thm E main}. Let $g\in S_2(\Gamma_0(N))$
be the newform associated to $E$ (so $L(E,s) = L(g,s)$). 
Then, in the notation of \ref{newform-gal}, $\sO_0=\sO=\BZ$, $\fkp = (p)$, $V_0=V=E[p]$, and $\ov\rho=\ov\rho_{E,p}$. 
Also, since $\bar\rho_{E,p}$ is irreducible, without loss of generality we may assume that $E$ is an optimal curve, so $A_0=A=E$.
To prove Theorem \ref{thm E main} it then suffices to show that the conditions (a)--(f) of Theorem \ref{Kconj thm} hold for these. 

Conditions (a), (b), and (e) of Theorem 
\ref{Kconj thm} are immediate from the hypotheses of Theorem \ref{thm E main}. 

To see that Hypothesis $\heart$ holds for $(g,p,K)$, in light of the assumption
that Hypothesis~$\spade$ holds for $(E,p,K)$, we need only check that part (4) of $\heart$ holds.  This is explained in \cite[Lem.~5.1(2)]{Z13}.
Thus part (4) of $\heart$ holds (with $N^-$ a product of an even number of primes), and so condition (c) of Theorem \ref{Kconj thm} holds.

To see that Hypothesis $\club$ holds for $\ov\rho$, we first note that $\ov\rho$ is irreducible by hypothesis. So part (1) of $\club$ holds. 
Since $\Ram(\ov\rho)\neq 0$ by
hypothesis (see part (3) of $\spade$), there is some prime $\ell\mid\mid N$, $\ell\neq p$, such that $\ov\rho$ is ramified at $\ell$. Since $\ell\mid\mid N$,
$E$ has multiplicative reduction at $\ell$ and so the action of $I_\ell$ on $\Ta_pE\cong\BZ_p^2$ is through a unipotent subgroup of
$\GL_2(\BZ_p)$. It follows that the image of $\ov\rho(I_\ell)$, which is non-zero, is also unipotent. In particular, $\ov\rho(G_\BQ)$ contains an element of 
order $p$. It then follows from \cite[Prop.~15]{Serre-Gal} that $\ov\rho(G_\BQ)$ contains $\SL_2(\BF_p)$ and hence equals $\GL_2(\BF_p)$. 
Since $p\geq 5$, it follows that part (2) of $\club$ also holds. Thus condition (d) of Theorem \ref{Kconj thm} holds.

Finally, we recall that by Lemma \ref{Linv lem EC}, hypothesis (c) of Theorem \ref{thm E main} implies that Hypothesis~$\fkL$ holds for $\ov\rho$
if $a(p)=1$. 
So condition (f) of Theorem \ref{Kconj thm} also holds. 

This completes the verification that the hypotheses of Theorem \ref{thm E main} imply that all the conditions of Theorem \ref{Kconj thm} hold for $g$, $p$ and $K$, and
so Theorem \ref{thm E main} follows.

\subsection{A theorem about $\ord(\kappa^\infty)$}
Recall that 
$$
\kappa^\infty = \{c_M(n) \in H^1(K,E[p^M]) \ : \ n\in \Lambda, M\leq M(n)\},
$$
the Kolyvagin system associated with $E$, $p$, and $K$ as in \ref{Koly Conj}. Let $\ord(\kappa^\infty)$ be the minimum of the number of prime factors of all integers $n\in \Lambda$
such that $c_M(n)\neq 0$ for some $M\leq M(n)$. Combining Theorem \ref{thm E main} with \cite[Thm.~4]{Koly-Selmer} we obtain an analog of \cite[Thm.~1.2]{Z13}:

\begin{thm}\label{thm E rank} 
Let $E$, $p$, and $K$ satisfy the hypotheses of Theorem \ref{thm E main}. Let $r_p^\pm$ be the $\BZp$-corank of
$\Sel_{p^\infty}(E/K)^\pm$, where the superscript `$\pm$' denotes the subgroup on which $\Gal(K/\BQ)$ acts as $\pm 1$. Then 
$$
\ord(\kappa^\infty) = \min\{ r_p^+,r_p^-\} -1.
$$
\end{thm}

\subsection{Proof of Theorem \ref{thm E BSD}}
Theorem \ref{thm E BSD} follows from Theorem \ref{thm E main} (and the slightly stronger statement from Theorem \ref{Kconj thm} that $M_\infty(g)=0$), 
by the same arguments used to deduce \cite[Thm.~10.2]{Z13} and \cite[Thm.~10.3]{Z13} 
(making use of Theorem \ref{thm BSD 0} and the period comparison in Lemma \ref{can period twist}).

\subsection{Proof of Theorem \ref{thm E/Q}}
Theorem \ref{thm E/Q} is deduced from the Theorem \ref{thm E rank} just as \cite[Thm.~1.4]{Z13} is deduced from \cite[Thm.~1.2]{Z13}: by making a good choice of $K$.

\end{document}